\documentclass[reqno,11pt]{amsart}
\usepackage{amsmath,amsfonts,amssymb,amsxtra,latexsym,amscd,enumerate,amsthm,verbatim}

\usepackage[margin=1.40in]{geometry}
\numberwithin{equation}{section}

\def\d{\delta}
\def\b{\beta}

\def\p{\partial}

\def\ep{\epsilon}

\def\O{\Omega}

\def\Z{\mathbb{Z}}

\def\D{\mathcal{D}}

\def\V{\mathcal{V}}

\def\H{\mathcal{H}}

\def\L{\mathcal{L}}

\def\Z{\mathbb{Z}}

\def\b{{\beta}}

\def\d{\delta}

\def\ep{\epsilon}
\def\eps{\epsilon}

\def\varep{\varepsilon}

\def\phii{\widetilde{\varphi}}

\def\al{\alpha}

\newtheorem{theorem}{Theorem}[section]
\newtheorem{lemma}[theorem]{Lemma}

\newtheorem{proposition}[theorem]{Proposition}
\newtheorem{definition}[theorem]{Definition}
\newtheorem{remark}[theorem]{Remark}

\def\beq{\begin{equation}}
\def\eeq{\end{equation}}

\def\beq{\begin{equation}}
\def\eeq{\end{equation}}

\begin{document}

\title[The one-fluid Euler-Maxwell system in 3D with vorticity]{Long term regularity of the one-fluid Euler-Maxwell system in 3D with vorticity}

\author{Alexandru D. Ionescu}
\address{Princeton University}
\email{aionescu@math.princeton.edu}

\author{Victor Lie}
\address{Purdue University}
\email{vlie@purdue.edu}

\thanks{The first author was supported in part by NSF grant DMS-1265818. The second author was supported in part by NSF grant DMS-1500958.}

\begin{abstract}
A basic model for describing plasma dynamics is given by the ``one-fluid'' Euler-Maxwell system, in which a compressible electron fluid interacts with its own self-consistent electromagnetic field. In this paper we prove long-term regularity of solutions of this system in 3 spatial dimensions, in the case of small initial data with nontrivial vorticity. 

Our main conclusion is that the time of existence of solutions depends only on the size of the vorticity of the initial data, as long as the initial data is sufficiently close to a constant stationary solution. 
\end{abstract}

\maketitle

\setcounter{tocdepth}{1}

\tableofcontents

\date{\today}

\maketitle

\section{Introduction}

A plasma is a collection of fast-moving charged particles and is one of the four fundamental states of matter. Plasmas are the most common phase of ordinary matter in the universe, both by mass and by volume. Essentially, all of the visible light from space comes from stars, which are plasmas with a temperature such that they radiate strongly at visible wavelengths. Most of the ordinary (or baryonic) matter in the universe, however, is found in the intergalactic medium, which is also a plasma, but much hotter, so that it radiates primarily as X-rays. We refer to \cite{Bit,DelBer} for physics references in book form.

One of the basic models for describing plasma dynamics is the Euler-Maxwell  \textquotedblleft two-fluid\textquotedblright\ model, in which two compressible ion and electron fluids interact with their own self-consistent electromagnetic field. In this paper we consider a slightly simplified version, the so-called one-fluid Euler-Maxwell system (EM) for electrons, which accounts for the interaction of electrons and the electromagnetic field, but neglects the dynamics of the ion fluid. The model describes the dynamical evolution of the functions $n_e:\mathbb{R}^3\to\mathbb{R}$ (the density of the fluid), $v_e:\mathbb{R}^3\to\mathbb{R}^3$ (the velocity field of the fluid), and $E',B':\mathbb{R}^3\to\mathbb{R}^3$ (the electric and magnetic fields), which evolve according to the coupled nonlinear system
\begin{equation}\label{initsyst}
\begin{cases}
&\partial_tn_e+\hbox{div}(n_ev_e)=0,\\
&m_e(\partial_tv_e+v_e\cdot\nabla v_e)=-P_e\nabla n_e-e\left[E'+(v_e/c)\times B'\right],\\
&\partial_tE'-c\nabla\times B'=4\pi en_ev_e,\\
&\partial_tB'+c\nabla\times E'=0,\\
\end{cases}
\end{equation}
together with the constrains
\begin{equation}\label{constr}
\hbox{div}(B')=0,\quad \hbox{div}(E')=-4\pi e(n_e-n^0).
\end{equation}
The constraints \eqref{constr} are propagated by the flow if they are satisfied at the initial time.

There are several physical constants in the above system: $-e<0$ is the electron's charge, $m_e$ is the electron's mass, $c$ denotes the speed of light, and $P_e$ is related to the effective electron temperature (that is $k_BT_e=n^0P_e$, where $k_B$ is the Boltzmann constant). In the system above we have chosen, for simplicity, the quadratic adiabatic pressure law $p_e=P_en_e^2/2$.

The system has a family of equilibrium solutions $(n_e,v_e,E',B')=(n^0,0,0,0)$, where $n^0>0$ is a constant. Our goal here is to investigate the long-term stability properties of these solutions.

\subsection{The main theorem} The system \eqref{initsyst}--\eqref{constr} is a complicated coupled nonlinear system of ten scalar evolution equations and two constraints. To simplify it, we make first linear changes of variables to normalize the constants. More precisely, let
\begin{equation*}
\lambda:=\frac{1}{c}\sqrt{\frac{4\pi e^2n^0}{m_e}},\qquad\beta:=\sqrt{\frac{4\pi e^2n^0}{m_e}},\qquad \alpha:=\frac{\lambda m_e c^2}{e}=\frac{4\pi en^0}{\lambda},\qquad d:=\frac{P_en^0}{m_ec^2}>0,
\end{equation*}
and define the functions $n,v,E,B$ by
\begin{equation*}
\begin{split}
&n_e(x,t)=n^0[1+n(\lambda x,\beta t)],\qquad v_e(x,t)=c\cdot v(\lambda x,\beta t),\\
&E'(x,t)=\alpha E(\lambda x,\beta t),\qquad\,\qquad\, B'(x,t)=\alpha B(\lambda x,\beta t).
\end{split}
\end{equation*}
The system \eqref{initsyst}--\eqref{constr} becomes

\beq \label{systI}
\begin{cases}
\partial_tn+\hbox{div}((1+n)v)&=0,\\
\partial_tv+v\cdot\nabla v+d\nabla n+E+v\times B&=0,\\
\partial_tE-\nabla\times B-(1+n)v&=0,\\
\partial_tB+\nabla\times E&=0,
\end{cases}
 \eeq
and
\beq \label{constr1}
\hbox{div}(B)=0,\qquad \hbox{div}(E)+n=0.
\eeq
The system depends only on the parameter $d$ in the second equation. In the physically relevant case we have $d\in(0,1)$, which we assume from now on.

We now define the \textit{vorticity} of our system (allowed to be nontrivial) as
\begin{equation}\label{vort}
Y:=B-\nabla\times v.
\eeq
We note that the system \eqref{systI} admits a conserved energy, defined by
\begin{equation}\label{EnCons}
\mathcal{E}_{conserved}:=\int_{\mathbb{R}^3}\big\{d|n|^2+(1+n)|v|^2+|E|^2+|B|^2\big\}\,dx.
\end{equation}

To state our main theorem we need to introduce some notation.

\begin{definition}\label{OneDef}   We define the rotational vector-fields,
\beq\label{difop}
\O_1:=x_2\p_3-x_3\p_2,\qquad\O_2:=x_3\p_1-x_1\p_3,\qquad\O_3:=x_1\p_2-x_2\p_1.
\eeq
For $m\geq 0$ let $\mathcal{V}_m$ denote the set of differential operators of the form
\beq\label{coordrotm}
\mathcal{V}_m:=\{\partial_1^{\alpha_1}\partial_2^{\alpha_2}\partial_3^{\alpha_3}\O_1^{\beta_1}\O_2^{\beta_2}\O_3^{\beta_3}:\alpha_1+\alpha_2+\alpha_3+\beta_1+\beta_2+\beta_3\leq m\}.
\eeq
For $N\geq 1$ and $p\in[1,\infty]$ we define the spaces $\H^{N}(\mathbb{R}^3)$ and $\mathcal{W}^{N,p}(\mathbb{R}^3)$ by the norms
\begin{equation}\label{alx1}
\|u\|_{\H^{N}(\mathbb{R}^3)}:=\sum_{\mathcal{L}\in \mathcal{V}_N}\|\L u\|_{L^2(\mathbb{R}^3)},\qquad \|u\|_{\mathcal{W}^{N,p}(\mathbb{R}^3)}:=\sum_{\mathcal{L}\in \mathcal{V}_N}\|\L u\|_{L^p(\mathbb{R}^3)}.
\end{equation}
For $N\geq 1$ as above, we let $\widetilde{\H}^N$ be the normed space
\begin{equation}\label{alx2}
\begin{split}
\widetilde{\H}^N:=\{&(n,v,E,B):\mathbb{R}^3\to\mathbb{R}\times\mathbb{R}^3\times\mathbb{R}^3\times\mathbb{R}^3:\\
&\|(n,v,E,B)\|_{\widetilde{\H}^N}:=\|n\|_{\H^N}+\|v\|_{\H^N}+\|E\|_{\H^N}+\|B\|_{\H^N}<\infty\}.
\end{split}
\end{equation}
\end{definition}

The following theorem in the main result of this paper:

\begin{theorem}\label{MainThm}
Assume $d\in(0,1)$, and let $N_0:=100$, $N_1:=N_0/2+2$, and $\beta:=10^{-6}$. Then there is a constant $\bar{\ep}=\bar{\ep}(d)>0$ with the following property: assume that $(n_0,v_0,E_0,B_0):\mathbb{R}^3\to\mathbb{R}\times\mathbb{R}^3\times\mathbb{R}^3\times\mathbb{R}^3$ are small, smooth, and localized initial data, i.e.
\begin{equation}\label{smallin}
\|(n_0,v_0,E_0,B_0)\|_{\widetilde{\H}^{N_0}}+\|(1+|x|^2)^{(1+\beta)/2}(1-\Delta)^3(n_0,v_0,E_0,B_0)\|_{\H^{N_1}}\leq \bar{\ep},
\end{equation}
satisfying the compatibility conditions
\beq \label{compatin}
\hbox{div}(B_0)=0,\qquad \hbox{div}(E_0)+n_0=0.
\eeq
Assume that the initial vorticity $Y_{0}=B_0-\nabla\times v_{0}$ satisfies the additional smallness condition
\begin{equation}\label{smallvort}
\|(1+|x|^2)^{1/4}Y_0\|_{\mathcal{H}^{N_1}}\leq \d_{0}\leq\bar{\ep}.
\end{equation}
Then there exists a unique solution $(n,v,E,B)\in C([0,T_{\delta_0}]: \widetilde{\H}^{N_0})$ of the system \eqref{systI}--\eqref{constr1} having the initial data $(n_0,v_0,E_0,B_0)$, where
\begin{equation}\label{Alx4}
T_{\delta_0}=\bar{\ep}/\d_0.
\end{equation}
\end{theorem}

\begin{remark}\label{Alx0}

(i) The main conclusion of the theorem is that the solutions extend and stay smooth at least up to time $T_{\delta_0}\gtrsim 1/\delta_0$, which depends only on the size $\delta_0$ of the vorticity of the initial data. Notice that this implies global regularity in the irrotational case $\delta_0=0$, thus providing a quantitative version of the  earlier theorems of \cite{GeMa} and \cite{IoPa2}.

(ii) One can derive more information about the solution $(n,v,E,B)$ of the system. For example, the solution satisfies the uniform bounds, for all $t\in[0,T_{\delta_0}]$,
\begin{equation*}
\|(n(t),v(t),E(t),B(t))\|_{\widetilde{\H}^{N_0}}\lesssim\overline{\ep},\qquad \|(1+|x|^2)^{1/4}Y(t)\|_{\mathcal{H}^{N_1}}\lesssim \d_{0},
\end{equation*}
where $Y(t)=B(t)-\nabla\times v(t)$. Moreover, the solution decouples into a superposition of two dispersive components $U_e$ and $U_b$ which propagate with different group velocities and decay, and a vorticity component $Y$, which is essentially transported by the flow. The two dispersive components can be studied precisely using the $Z$-norm, see Definition \ref{MainZDef}.
\end{remark}

\subsection{Previous work on long-term regularity} The local regularity theory of the Euler--Maxwell system follows easily by energy estimates. The question of long-term regularity is much more interesting and has been studied in several recent papers. 

The dynamics of the full Euler--Maxwell system is extremely complex, due to a large number of coupled interactions and many types of resonances. Even at the linear level, there are ion-acoustic waves, Langmuir waves, light waves etc. At the nonlinear level, the Euler--Maxwell system is the origin of many well-known dispersive PDE's which can be derived via scaling and asymptotic expansions. See also the introduction of \cite{GuIoPa} for a longer discussion of the Euler--Maxwell system in 3D, and its connections to many other models in mathematical physics, such as the Euler--Poisson model, the Zakharov system, the KdV, the KP, and the NLS.

Because of this complexity it is natural to study first simplified models, such as the one-fluid Euler--Poisson model (first studied by Guo \cite{Guo}) and the one-fluid Euler--Maxwell system (which is the system \eqref{initsyst}). In particular, the one-fluid Euler--Maxwell system shares many of the features and the conceptual difficulties of the full system, but is simpler at the analytical level. Under suitable irrotationality assumptions, this system can be reduced to a coupled system of two Klein--Gordon equations with different speeds and no null structure. While global results are classical
in the case of scalar wave and Klein--Gordon equations, see for example \cite{Jo,JoKl, Kl2, KlVf,  Kl, Kl4, Ch, Sh, Si, DeFa, DeFaXu, Alin2, Alin3}, it was pointed out by Germain \cite{Ge} that there are key 
new difficulties in the case of a coupled system of Klein--Gordon equations with different speeds. In this case, the classical vector-field  method does not 
seem to work well, and there are large sets of resonances that contribute in the analysis. Global regularity for small irrotational solutions of this model was proved by Germain--Masmoudi \cite{GeMa} and Ionescu--Pausader \cite{IoPa2}, using more subtle arguments based on Fourier analysis. 

In 3 dimensions, nontrivial global solutions of the full two-fluid system were 
constructed for the first time by Guo--Ionescu--Pausader \cite{GuIoPa} (small irrotational perturbations of constant solutions), 
following the earlier partial results in simplified models in \cite{Guo,GuPa,GeMa,IoPa2}. 

The one-fluid Euler--Poisson system and the one-fluid Euler--Maxwell system have also been studied in 2 dimensions, where the global results are harder due to less dispersion and slower decay. See \cite{IoPa1}, \cite{LiWu}, and \cite{DeIoPa}.

\subsubsection{Nontrivial vorticity} We remark that all the global regularity results described above are restricted to the case of solutions with trivial vorticity. This is also the case with the global regularity results in many other quasilinear fluid models, such as water waves, see the introduction of \cite{DeIoPaPu} for a longer discussion. 

In fact, all proofs of global existence in quasilinear evolutions depend in a crucial way on establishing quantitative decay of solutions over time. On the other hand, one usually expects that vorticity is transported by the flow and does not decay. This simple fact causes a serious obstruction to proving global existence for solutions with dynamically nontrivial vorticity.  

In this paper we would like to initiate the study of long-term regularity of solutions with nontrivial vorticity. However, we are not able to establish the global existence of such solutions for any of the Euler-Maxwell or Euler--Poisson systems. Instead we prove that sufficiently small solutions extend smoothly on a time of existence that depends only on the size of the vorticity. 

Such a theorem can be interpreted as a quantitative version of global regularity theorems for small solutions with trivial vorticity described earlier. In fact, our Theorem \ref{MainThm} immediately implies the global regularity theorems of \cite{GeMa} and \cite{IoPa2}, simply by letting $\delta_0\to 0$. 

An important consideration to keep in mind is the length of the time of existence of solutions. In our case we show that this time of existence is at least $c/\delta_0$, where $\delta_0$ is the size of the vorticity component of the initial data, and $c$ is a small constant. This is consistent with the time of existence of the simple equation
\begin{equation}\label{MoVo}
\partial_{t}Y=Y^2.
\end{equation}
One can think of this equation as a model for the vorticity equation, in dimension 3, which ignores all the other interactions and the precise structure of the vorticity equation. The $c/\delta_0$ time of existence appears to be quite robust, and one can hope to prove a theorem like Theorem \ref{MainThm} in other models in which global regularity for solutions with trivial vorticity is known.  

One might also hope that more involved analysis would allow one to extend solutions beyond the $c/\delta_0$ time of existence, particularly in certain models in dimension 2 when the vorticity equation is known to behave better than the simple equation \eqref{MoVo}. We hope to return to such issues in the future. 

\subsection{Main ideas of the proof}\label{MainIdea}The classical mechanism to establish long-term regularity for quasilinear equations has two main components:

\setlength{\leftmargini}{1.8em}
\begin{itemize}
  \item[(1)] Control of high frequencies (high order Sobolev norms);
\smallskip
  \item[(2)] Dispersion/decay of the solution over time.
\end{itemize}

The interplay of these two aspects has been present since the seminal work of Klainerman \cite{Kl2}--\cite{Kl4}, Christodoulou \cite{Ch}, and Shatah \cite{Sh}. In the last few years new methods have emerged in the study of global solutions of quasilinear evolutions, inspired by the advances in semilinear theory.
The basic idea is to combine the classical energy and vector-fields methods with refined analysis of the Duhamel formula, using the Fourier transform.
This is the essence of the ``method of space-time resonances'' of Germain--Masmoudi--Shatah \cite{GeMaSh,GeMaSh2}, see also Gustafson--Nakanishi---Tsai \cite{GuNaTs},
and of the refinements in \cite{IoPa1,IoPa2,GuIoPa,GuIoPa2,DeIoPa,De,DeIoPaPu}, using atomic decompositions and sophisticated norms.

This general framework needs to be adapted to our case, where we have non-decaying components and we are aiming for a lifespan that depends only on the size of these components. To illustrate the main ideas, consider the following schematic system
\begin{equation}\label{scoo}
\begin{split}
(\partial_t+i\Lambda)U&=O(U^2)+O(UY)+O(Y^2),\\
\partial_t Y&=O(UY)+O(Y^2).
\end{split}
\end{equation}
Here one should think of $U$ as generic dispersive variables (take for instance the Klein--Gordon case $\Lambda=\sqrt{1-\Delta}$) and $Y$ represent generic non-dispersive vorticity-type components. The nonlinearities $O(U^2), O(UY), O(Y^2)$ are to be thought of as generic quadratic nonlinearities that may lose derivatives. See \eqref{KG} for the precise system in our case, keeping in mind that there are two types of dispersive variables corresponding to two different speeds of propagation. 

Our analysis of solutions of such a system contains three main ingredients:

\begin{itemize}
\item Energy estimates for the full system. These estimates allow us to control high Sobolev norms and weighted norms (corresponding to the rotation vector-field) of the solution. They are not hard in our case, since we are able to prove independently $L^1_t$ pointwise control of the solution.

\item Vorticity energy estimates. This is a new ingredient in our problem. We need to show that the vorticity stays small, that is $\lesssim\delta_0$, on the entire time of existence. These estimates depend again on the $L^1_t$ pointwise control of the solution and on the structure of the nonlinearity of the vorticity equation (without a $O(U^2)$ term).

\item Dispersive analysis. The dispersive estimates, which lead to decay, rely on a bootstrap argument in a suitable $Z$ norm. The norm we use here is similar to the $Z$ norm introduced in the 2D problem in \cite{DeIoPa} and accounts for the rotation invariance of the system. We analyze carefully the Duhamel formula for the first equation in \eqref{scoo}, in particular the quadratic interactions related to the set of resonances. The analysis of the terms $O(Y^2)$ and $O(YU)$, which contain the $\mathrm{transport}\times\mathrm{transport}\to\mathrm{dispersive}$ and the $\mathrm{transport}\times\mathrm{dispersive}\to\mathrm{dispersive}$ interactions, is new, when compared to the irrotational global results described earlier such as \cite{IoPa2}. On the other hand, the analysis of the term $O(U^2)$, which involves a large set of space-time resonances, due to the two different speeds of propagation, has similarities with the analysis in \cite{IoPa1,IoPa2,GuIoPa,GuIoPa2}. 
\end{itemize}

At the implementation level, we remark that we are able to completely decouple the decay parameter $\beta$, which can be taken very small, see Definition \ref{MainZDef}, from the smoothness parameters $N_0$ and $N_1$. These parameters were related to each other in earlier work, such as \cite{IoPa1,IoPa2, GuIoPa,GuIoPa2}. 
As a result, we are able to reduce substantially the total number of derivatives $N_0$ and $N_1$ in the main theorem.\footnote{These smoothness parameters can be further reduced by longer and more careful analysis, but our goal here is just to demonstrate that these parameters can be decoupled from the decay parameters in the $Z$ norm.}

\subsection{Organization} The rest of the paper is organized as follows: in section \ref{prelims} we introduce most of the key definitions, such as the $Z$ norm, rewrite our main system as a dispersive system for the quasilinear variables (diagonalized at the linear level), and state the main bootstrap proposition. In section \ref{lemmas} we summarize some lemmas that are being used in the rest of the paper, mostly concerning linear analysis and the resonant structure of the oscillatory phases. In section \ref{EneEst} we prove our main energy estimates, both for the full energy of the system and for the vorticity energy. Finally, in sections \ref{ParT}--\ref{DispInter} we prove our main dispersive estimates for the decaying components of the solution.

\section{Preliminaries}\label{prelims}

In this section we rewrite our main system as a quasilinear dispersive system (diagonalized at the linear level), summarize the main definitions, and state the main bootstrap proposition.

\subsection{Diagonalization} We assume that $(n,v,E,B)$ satisfy the system of equations \eqref{systI}--\eqref{constr1} and use the Hodge decomposition. Let
\begin{equation}\label{Alx11}
\begin{split}
&F:=|\nabla|^{-1}\hbox{div}(v),\qquad \,G:=|\nabla|^{-1}\nabla\times v,\\
&Z:=|\nabla|^{-1}\hbox{div}(E),\qquad W:=|\nabla|^{-1}\nabla\times E,\qquad Y=B-\nabla\times v.
\end{split}
\end{equation}
Let $R_j:=|\nabla|^{-1}\partial_j$ denote the Euclidean Riesz transforms. Then we can express the variables $n,v,E,B$ elliptically, in terms of $F,G,Z,W,Y$, according to the formulas
\begin{equation}\label{Alx12}
v_k=-R_kF+\in_{jlk}R_jG_l,\quad E_k=-R_kZ+\in_{jlk}R_jW_l,\quad n=-|\nabla|Z,\quad B=Y+|\nabla|G.
\end{equation}
Recall also that
\begin{equation*}
\hbox{div}(Y)=0,\qquad \hbox{div}(G)=0,\qquad \hbox{div}(W)=0.
\end{equation*}

By taking divergences and curls, the system \eqref{systI} gives the evolution equations
\begin{equation}\label{Alx13}
\begin{cases}
\partial_t F+(1+d|\nabla|^2)Z&=-R\cdot(v\cdot \nabla v)-R\cdot (v\times B),\\
\partial_t G+W&=-R\times(v\cdot \nabla v)-R\times (v\times B),\\
\partial_t Z-F&=R\cdot (nv),\\
\partial_t W-(1+|\nabla|^2)G-|\nabla|Y&=R\times (nv),\\
\partial_t Y&=|\nabla|\big[R\times(v\cdot \nabla v)+R\times (v\times B)\big].
\end{cases}
\end{equation}
Since $B=Y+\nabla\times v$ and $v\times(\nabla\times v)=\nabla(|v|^2/2)-v\cdot\nabla v$ we have
\begin{equation}\label{Alx14}
\begin{split}
R\cdot (v\times B)&=R\cdot (v\times Y)-|\nabla|(|v|^2)/2-R\cdot(v\cdot \nabla v),\\
R\times(v\times B)&=R\times(v\times Y)-R\times(v\cdot \nabla v).
\end{split}
\end{equation}

Let
\begin{equation}\label{Alx15}
\begin{split}
&U_e:=\Lambda_eZ+iF,\qquad\qquad\qquad\qquad\Lambda_e:=\sqrt{1+d|\nabla|^2},\\
&U_b:=W+i\Lambda_bG+i\Lambda_b^{-1}|\nabla|Y,\qquad\, \Lambda_b:=\sqrt{1+|\nabla|^2}.
\end{split}
\end{equation}
The formulas above show that
\beq \label{KG}
\begin{cases}
(\partial_t+i\Lambda_e) U_e&= \Lambda_e (R\cdot [n v])+i |\nabla| (|v|^2)/2-i R\cdot (v\times Y),\\
(\partial_t+i\Lambda_b) U_b&= R\times [nv]- i \Lambda_b^{-1} R\times (v\times Y),\\
\partial_t Y&=\nabla\times (v\times Y)\,.
\end{cases}
 \eeq

Conversely, the physical variables $n,v,E,B$ can be recovered from the dispersive variables $U_e,U_b,Y$ by the formulas, see \eqref{Alx12},
\begin{equation}\label{Alx17}
\begin{split}
&n=-|\nabla|Z,\qquad v=-RF+R\times G,\qquad E=-RZ+R\times W,\qquad B=Y+|\nabla|G,\\
&F=\Im(U_e),\qquad G=\Lambda_b^{-1}\Im(U_b)-\Lambda_b^{-2}|\nabla|Y,\qquad Z=\Lambda_e^{-1}\Re(U_e),\qquad W=\Re(U_b).
\end{split}
\end{equation}
The formulas show that the sets of variables $(n,v,E,B,Y)$ and $(U_e,U_b,Y)$ are elliptically equivalent, for example, for any $m\geq 1$
\begin{equation}\label{Alx18}
\|n\|_{\mathcal{H}^m}+\|v\|_{\mathcal{H}^m}+\|E\|_{\mathcal{H}^m}+\|B\|_{\mathcal{H}^m}+\|Y\|_{\mathcal{H}^m}\approx \|U_e\|_{\mathcal{H}^m}+\|U_b\|_{\mathcal{H}^m}+\|Y\|_{\mathcal{H}^m}.
\end{equation}

\subsection{Main notations and definitions}\label{NotDef}

\subsubsection{Littlewood--Paley projections} We fix $\varphi:\mathbb{R}\to[0,1]$ an even smooth function supported in $[-8/5,8/5]$ and equal to $1$ in $[-5/4,5/4]$. Let
\begin{equation*}
\begin{split}
&\varphi_k(x):=\varphi(|x|/2^k)-\varphi(|x|/2^{k-1})\qquad\text{ for any }k\in\mathbb{Z},\,x\in\mathbb{R}^3,\qquad\\ &\varphi_I:=\sum_{m\in I\cap\mathbb{Z}}\varphi_m\text{ for any }I\subseteq\mathbb{R}.
\end{split}
\end{equation*}
For any $B\in\mathbb{R}$ let
\begin{equation*}
\varphi_{\leq B}:=\varphi_{(-\infty,B]},\quad\varphi_{\geq B}:=\varphi_{[B,\infty)},\quad\varphi_{<B}:=\varphi_{(-\infty,B)},\quad \varphi_{>B}:=\varphi_{(B,\infty)}.
\end{equation*}
For any $a<b\in\mathbb{Z}$ and $j\in[a,b]\cap\mathbb{Z}$ let
\begin{equation}\label{Alx80}
\varphi^{[a,b]}_j:=
\begin{cases}
\varphi_j\qquad&\text{ if }a<j<b,\\
\varphi_{\leq a}\qquad&\text{ if }j=a,\\
\varphi_{\geq b}\qquad&\text{ if }j=b.
\end{cases}
\end{equation}

For any $x\in\mathbb{R}$ let $x^+:=\max(x,0)$, $x^-:=\min(x,0)$. Let
\begin{equation*}
\mathcal{J}:=\{(k,j)\in\mathbb{Z}\times\mathbb{Z}:\,j\geq \max(-k,0)\}.
\end{equation*}
For any $(k,j)\in\mathcal{J}$ let
\begin{equation*}
\phii^{(k)}_j(x):=
\begin{cases}
\varphi_{(-\infty,\max(-k,0)]}(x)\quad&\text{ if }\,\,j=\max(-k,0),\\
\varphi_j(x)\quad&\text{ if }\,\,j\geq 1+\max(-k,0).
\end{cases}
\end{equation*}
and notice that, for any $k\in\mathbb{Z}$ fixed,
\begin{equation*}
\sum_{j\geq \max(-k,0)}\phii^{(k)}_j=1.
\end{equation*}
For any interval $I\subseteq\mathbb{R}$ let
\begin{equation*}
\phii^{(k)}_I(x):=\sum_{j\in I,\,(k,j)\in\mathcal{J}}\phii^{(k)}_j(x).
\end{equation*}

Let $P_k$, $k\in\mathbb{Z}$, denote the operator on $\mathbb{R}^3$ defined by the Fourier multiplier $\xi\to \varphi_k(\xi)$. Similarly, for any $I\subseteq \mathbb{R}$ let $P_I$ denote the operator on $\mathbb{R}^3$
 defined by the Fourier multiplier $\xi\to \varphi_I(\xi)$. For any $(k,j)\in\mathcal{J}$ let $Q_{jk}$ denote the operator
\begin{equation}\label{qjk}
(Q_{jk} f)(x):=\phii^{(k)}_j(x)\cdot P_kf(x).
\end{equation}

\subsubsection{Phases, linear profiles, and the $Z$-norm}

An important role will be played by the profiles $V_e,V_b$ defined by
\begin{equation}\label{variables4}
V_e(t):=e^{it\Lambda_e}U_e(t),\qquad V_b(t):=e^{it\Lambda_b}U_b(t),
\end{equation}
where $U_e$ and $U_b$ are the dispersive variables defined in \eqref{Alx15}, and $\Lambda_e=\sqrt{1-d\Delta}$ and $\Lambda_b=\sqrt{1-\Delta}$ as before. We define
\begin{equation}
\begin{split}
&U_{-e}:=\overline{U_e},\qquad U_{-b}:=\overline{U_b};\qquad V_{-e}:=\overline{V_e},\qquad V_{-b}:=\overline{V_b};\\
&\Lambda_{-e}:=-\Lambda_{e},\qquad\Lambda_{-b}:=-\Lambda_b.
\end{split}
\label{notation}\end{equation}
Let \begin{equation}\label{symbol0}\mathcal{P}:=\{e,b,-e,-b\}.\end{equation}
For $\sigma,\mu,\nu\in \mathcal{P}$, we define the associated phase function
\begin{equation}\label{phasedef}
\Phi_{\sigma\mu\nu}(\xi,\eta):=\Lambda_{\sigma}(\xi)-\Lambda_{\mu}(\xi-\eta)-\Lambda_{\nu}(\eta),
\end{equation}
and the corresponding function
\begin{equation} \label{deflambd}
\begin{split}
&\Phi^{+}_{\sigma\mu\nu}(\alpha,\beta):=\Phi_{\sigma\mu\nu}(\alpha e,\beta e)=\lambda_{\sigma}(\alpha)-\lambda_{\mu}(\alpha-\beta)-\lambda_{\nu}(\beta),\\
&\lambda_e(r)=-\lambda_{-e}(r):=\sqrt{1+dr^2},\qquad\lambda_b(r)=-\lambda_{-b}(r):=\sqrt{1+r^2},
\end{split}
\end{equation}
where $e\in\mathbb{S}^{1}$ and $\alpha,\beta\in\mathbb{R}$.
If $(\mu,\nu)\in\mathcal{P}\times\mathcal{P}\setminus\{(e,-e),(-e,e),(b,-b),(-b,b)\}$, by Proposition \ref{spaceres} for any $\xi\in\mathbb{R}^2$ there exists a unique $\eta=p(\xi)\in\mathbb{R}^2$ so that $(\nabla_{\eta}\Phi_{\sigma\mu\nu})(\xi,\eta)=0$ (a space resonance point). We define, for a sufficiently large constant $\D_0$ that depends only on the parameter $d\in(0,1)$,
\begin{equation}\label{psidag}
\Psi_{\sigma\mu\nu}(\xi):=\Phi_{\sigma\mu\nu}(\xi,p(\xi)),\qquad \Psi_{\sigma}^{\dagger}(\xi):=2^{\mathcal{D}_0}(1+|\xi|)\inf_{\mu,\nu\in\mathcal{P};\nu+\mu\neq 0}|\Psi_{\sigma\mu\nu}(\xi)|,
\end{equation}
and notice that these functions are radial. The functions $\Psi_e^{\dagger}$ and $\Psi_b^{\dagger}$ are described in Remark \ref{largeres}; in particular,  $\Psi_e^{\dagger}\geq 10$ while $\Psi_b^{\dagger}$ vanishes on two spheres $|\xi|=\gamma_{1,2}=\gamma_{1,2}(d)\in(0,\infty)$. These spheres correspond to space-time resonances. For $n\in\mathbb{Z}$ we define the operators $A_{n}^{\sigma}$ by
\begin{equation}\label{aop}
\widehat{A_{n}^{\sigma}f}(\xi):=\varphi_{-n}(\Psi_{\sigma}^{\dagger}(\xi))\cdot\widehat{f}(\xi),
\end{equation}
for $\sigma\in\{e,b\}$. Given an integer $j\geq 0$ we define the operators $A^\sigma_{n,(j)}$, $n\in\{0,\ldots,j+1\}$, by
\begin{equation*}
A_{0,(j)}^\sigma:=\sum _{n'\leq 0}A_{n'}^\sigma,\qquad A_{j+1,(j)}^\sigma:=\sum _{n'\geq j+1}A_{n'}^\sigma,\qquad A_{n,(j)}^\sigma:=A_{n}^\sigma\,\,\text{ if }\,\,0<n<j+1.
\end{equation*}

We are now ready to define the main $Z$-norm.

\begin{definition}\label{MainZDef} For $\sigma\in\{e,b\}$ we define
\begin{equation}\label{sec5}
Z_1^\sigma:=\{f\in L^2(\mathbb{R}^3):\,\|f\|_{Z_1^\sigma}:=\sup_{(k,j)\in\mathcal{J}}\|Q_{jk}f\|_{B^\sigma_{j}}<\infty\},
\end{equation}
where, with $\beta:=10^{-6}$,
\begin{equation}\label{znorm2}
\|g\|_{B_j^{\sigma}}:=\sup_{0\leq n\leq j+1}2^{(1+\beta)j-4\beta n}\|A_{n,(j)}^{\sigma}g\|_{L^{2}}.
\end{equation}
Finally, with $N_1=N_0/2+2$ as before, $\mathcal{V}_{N_1}$ as in \eqref{coordrotm}, and $D^\alpha=\partial_1^{\alpha_1}\partial_2^{\alpha_2}\partial_3^{\alpha_3}$, we define
\begin{equation}\label{znorm}
Z:=\big\{(f_e,f_b)\in L^2\times L^2:\,\|(f_e,f_b)\|_{Z}:=\sup_{\mathcal{L}\in \mathcal{V}_{N_1},\,|\alpha|\leq 4}\big[\|D^\alpha\mathcal{L}f_e\|_{Z_1^e}+\|D^\alpha\mathcal{L}f_b\|_{Z_1^b}\big]<\infty\big\}.
\end{equation}
\end{definition}

Notice that, when $\sigma=e$ we have the simpler formula,
\begin{equation*}
\|g\|_{B_j^{e}}\approx 2^{(1+\beta)j}\|g\|_{L^{2}}.
\end{equation*}
Similarly if $j\lesssim 1$ then $\|g\|_{B_j^{b}}\approx\|g\|_{L^{2}}$. The operators $A_{n,(j)}^{\sigma}$ are relevant only when $\sigma=b$ and $j\gg 1$, to localize to thin neighborhoods of the space-time resonant sets. The small factors $2^{-4\beta n}$ in \eqref{znorm2}, which are connected to the operators $A_{n,(j)}^{b}$, are important only in the space-time resonant analysis, in the proof of the bound \eqref{top1} in Lemma \ref{Reso01}.

\subsection{The main bootstrap proposition}\label{bootstrap0} Our main result is the following proposition:

\begin{proposition}\label{bootstrap} Suppose $(n,v,E,B)$ is a solution to \eqref{systI}--\eqref{constr1} on some time
interval $[0,T]$, $T\in[1,\bar{\eps}/\delta_0]$, with initial data $(n_0,v_0,E_0,B_0)$, and define $(V_e,V_b)$ as in \eqref{variables4} and $Y=B-\nabla\times v$. Assume that
\begin{equation}\label{bootstrap1}
\|(n_0,v_0,E_0,B_0)\|_{\widetilde{\H}^{N_0}}+\|(V_e(0),V_b(0))\|_{Z}\lesssim \bar{\eps}
\end{equation}
and
\begin{equation}\label{bootstrapV1}
\|(1+|x|^2)^{1/4}Y_0\|_{\mathcal{H}^{N_1}}\leq \d_{0}\leq\bar{\eps}.
\end{equation}
In addition, assume that for any $t\in[0,T]$,
\begin{equation}\label{bootstrap2}
\|(n(t),v(t),E(t),B(t))\|_{\widetilde{\H}^{N_0}}+\|(V_e(t),V_b(t))\|_{Z}\leq \overline{C}\bar{\eps}
\end{equation}
and
\begin{equation}\label{bootstrapV2}
\|(1+|x|^2)^{1/4}Y(t)\|_{\mathcal{H}^{N_1}}\leq \overline{C}\d_{0},
\end{equation}
for some sufficiently large constant $\overline{C}$. Then, for any $t\in[0,T]$,
\begin{equation}\label{bootstrap3}
\|(n(t),v(t),E(t),B(t))\|_{\widetilde{\H}^{N_0}}+\|(V_e(t),V_b(t))\|_{Z}\leq \overline{C}\bar{\eps}/2
\end{equation}
and
\begin{equation}\label{bootstrapV3}
\|(1+|x|^2)^{1/4}Y(t)\|_{\mathcal{H}^{N_1}}\leq \overline{C}\d_{0}/2.
\end{equation}
\end{proposition}

The constant $\overline{C}$ can be fixed sufficiently large, depending only on $d$, and the constant $\overline{\eps}$ is small relative to $1/\overline{C}$. Given Proposition \ref{bootstrap}, Theorem \ref{MainThm} follows using a local existence result and a continuity argument.
See \cite[Sections 2 and 3]{IoPa2} (in particular Proposition 2.2 and Proposition 2.4) for similar arguments.

The rest of this paper is concerned with the proof of Proposition \ref{bootstrap}. This proposition follows from
Proposition \ref{BootstrapEE1}, Proposition \ref{BootstrapEE2}, and Proposition \ref{BootstrapZNorm}.

\section{Some lemmas}\label{lemmas}

In this section we collect several lemmas that are used in the rest of the paper.  We fix a sufficiently large constant $\mathcal{D}\geq 10\mathcal{D}_0$.

\subsubsection{Integration by parts}\label{Ipa} We start with two lemmas that are used often in integration by parts arguments. See \cite[Lemma 5.4]{IoPa2} and \cite[Lemma ]{DeIoPa} for the proofs.

\begin{lemma}\label{tech5} Assume that $0<\eps\leq 1/\eps\leq K$, $N\geq 1$ is an integer, and $f,g\in C^{N+1}(\mathbb{R}^3)$. Then
\begin{equation}\label{ln1}
\Big|\int_{\mathbb{R}^3}e^{iKf}g\,dx\Big|\lesssim_N (K\eps)^{-N}\big[\sum_{|\alpha|\leq N}\eps^{|\alpha|}\|D^\alpha_xg\|_{L^1}\big],
\end{equation}
provided that $f$ is real-valued,
\begin{equation}\label{ln2}
|\nabla_x f|\geq \mathbf{1}_{{\mathrm{supp}}\,g},\quad\text{ and }\quad\|D_x^\alpha f \cdot\mathbf{1}_{{\mathrm{supp}}\,g}\|_{L^\infty}\lesssim_N\eps^{1-|\alpha|},\,2\leq |\alpha|\leq N+1.
\end{equation}
\end{lemma}

We will need another result about integration by parts using the rotation vector-fields $\Omega_j$. The lemma below (which is used only in the proof of the more technical Lemma \ref{Reso01}) follows from Lemma 3.8 in \cite{DeIoPa}.

\begin{lemma}\label{RotIBP}
Assume that $t\in[2^{m}-1,2^{m+1}]$, $m\geq 0$, $1\leq A\lesssim 2^m$, and
\begin{equation}\label{hypo}
\begin{split}
\Vert f\Vert_{\H^{20}}+\Vert g\Vert_{\H^{20}}+\sup_{0\leq\vert\alpha\vert\le N}A^{-\vert\alpha\vert}\Vert D^\alpha\widehat{f}\,\Vert_{L^2}&\le 1,\\
\sup_{\xi,\eta}\sup_{\vert\alpha\vert\le N}2^{-|\alpha|m/2}\vert D^\alpha_\eta n(\xi,\eta)\vert&\le 1.
\end{split}
\end{equation}
Assume that $\Phi=\Phi_{\sigma\mu\nu}$ for some $\sigma,\mu,\nu\in\{e,b,-e,-b\}$. For $\xi\in\mathbb{R}^3$ and $p\in[-m/2,0]$ let
\begin{equation*}
I^1_p(\xi):=\int_{\mathbb{R}^3}e^{it\Phi(\xi,\eta)}n(\xi,\eta)\varphi_p((\Omega_1)_\eta\Phi(\xi,\eta))\psi_1(\xi,\eta)\widehat{f}(\xi-\eta)\widehat{g}(\eta)d\eta,
\end{equation*}
where $(\Omega_1)_\eta=\eta_2\partial_{\eta_3}-\eta_3\partial_{\eta_2}$ is the rotation vector-field defined in \eqref{difop},
\begin{equation}\label{hypo5}
\psi_1(\xi,\eta):=\varphi_{\geq -\D}(\mathrm{Pr}_1(\xi))\varphi_{\geq -\D}(\mathrm{Pr}_1(\eta))\varphi_{\geq -\D}(\mathrm{Pr}_1(\xi-\eta))\cdot \varphi_{\leq \D}(\xi)\varphi_{\leq \D}(\eta)\varphi_{\leq\D}(\xi-\eta),
\end{equation}
and $\mathrm{Pr}_1:\mathbb{R}^3\to\mathbb{R}^2$, $\mathrm{Pr}_1(v_1,v_2,v_3):=(v_2,v_3)$. Then
\begin{equation}\label{OmIBP}
\vert I^1_p(\xi)\vert\lesssim_N  (2^{p}2^{m/2})^{-N}+(A2^{-m})^N+2^{-4m}.
\end{equation}
A similar bound holds for the integrals $I_p^2$ and $I_p^3$ obtained by replacing
the vector-field $\Omega_1$ with the vector-fields $\Omega_2$ and $\Omega_3$ respectively, and replacing the cutoff function $\psi_1$ with cutoff functions $\psi_2$ and $\psi_3$ respectively (defined as in \eqref{hypo5}, but with the projection $\mathrm{Pr}_1$ replaced by the projections $\mathrm{Pr}_2(v_1,v_2,v_3):=(v_1,v_3)$ and $\mathrm{Pr}_3(v_1,v_2,v_3):=(v_1,v_2)$ respectively). In addition, if $(1+\beta/20)\nu\geq -m$, then the same bounds hold when $I_p^j$, $j\in\{1,2,3\}$, are replaced by the integrals (notice the additional localization in modulation factor $\varphi_\nu(\Phi(\xi,\eta))$)
\begin{equation*}
\widetilde{I^j_p}(\xi):=\int_{\mathbb{R}^3}e^{it\Phi(\xi,\eta)}\varphi_\nu(\Phi(\xi,\eta))n(\xi,\eta)\varphi_p((\Omega_j)_\eta\Phi(\xi,\eta))\psi_j(\xi,\eta)\widehat{f}(\xi-\eta)\widehat{g}(\eta)d\eta.
\end{equation*}
\end{lemma}

\subsubsection{Linear and bilinear operators}\label{Lbl} To bound bilinear operators, we often use the following simple lemma.

\begin{lemma}\label{L1easy}
Assume $f_1,f_2,f_3\in L^2(\mathbb{R}^3)$, and $M:(\mathbb{R}^3)^2\to\mathbb{C}$ is a continuous compactly supported function. Then
\begin{equation}\label{ener62}
\Big|\int_{(\mathbb{R}^3)^2}M(\xi_1,\xi_2)\cdot\widehat{f_1}(\xi_1)\widehat{f_2}(\xi_2)\widehat{f_3}(-\xi_1-\xi_2)\,d\xi_1d\xi_2\Big|\lesssim \big\|\mathcal{F}^{-1}M\big\|_{L^1}\|f_1\|_{L^{p_1}}\|f_2\|_{L^{p_2}}\|f_3\|_{L^{p_3}},
\end{equation}
for any exponents $p_1,p_2,p_3\in[1,\infty]$ satisfying $1/p_1+1/p_2+1/p_3=1$. As a consequence
\begin{equation}\label{ener62.1}
\Big\|\mathcal{F}_{\xi}^{-1}\Big\{\int_{\mathbb{R}^3}M(\xi,\eta)\widehat{f_2}(\eta)\widehat{f_3}(-\xi-\eta)\,d\eta\Big\}\Big\|_{L^{q}}\lesssim \big\|\mathcal{F}^{-1}M\big\|_{L^1}\|f_2\|_{L^{p_2}}\|f_{3}\|_{L^{p_3}},
\end{equation}
if $q,p_2,p_3\in[1,\infty]$ satisfy $1/p_2+1/p_3=1/q$.
\end{lemma}

Our next lemma, which is also used to bound bilinear operators, shows that localization with respect to the phase is often a bounded operation. See \cite[Lemma 3.10]{DeIoPa} for the proof.

\begin{lemma}\label{PhiLocLem}
Let $s\in[2^{m}-1,2^m]$, $m\geq 0$, and $(1+\beta/20)p\geq -m$. With $\Lambda_0=0$ let\footnote{Notice that this is a slightly larger class of phases than those defined in section \ref{prelims}, i.e. it includes the contributions of the vorticity variables (corresponding to $\mu=0$ or $\nu=0$).}
\begin{equation}\label{zax1}
\Phi(\xi,\eta)=\Phi_{\sigma\mu\nu}(\xi,\eta)=\Lambda_\sigma(\xi)-\Lambda_\mu(\xi-\eta)-\Lambda_\nu(\eta),\qquad \sigma\in\mathcal{P},\,\mu,\nu\in\mathcal{P}\cup\{0\}.
\end{equation}
 Assume that $1/2=1/q+1/r$, $\chi$ is a Schwartz function, and $\|\mathcal{F}^{-1}(n)\|_{L^1(\mathbb{R}^3\times\mathbb{R}^3)}\leq 1$. Then
\begin{equation*}
\begin{split}
\Big\Vert \varphi_{\leq 10m}(\xi)\int_{\mathbb{R}^3}e^{is\Phi(\xi,\eta)}&\chi(2^{-p}\Phi(\xi,\eta))n(\xi,\eta)\widehat{f}(\xi-\eta)\widehat{g}(\eta)d\eta\Big\Vert_{L^2_\xi}\\
&\lesssim\sup_{t\in[s/10,10s]}\Vert e^{-it\Lambda_\mu}f\Vert_{L^q}\Vert e^{-it\Lambda_\nu}g\Vert_{L^r}+2^{-10m}\Vert f\Vert_{L^2}\Vert g\Vert_{L^2},
\end{split}
\end{equation*}
where the constant in the inequality only depends on the function $\chi$.
\end{lemma}

The nonlinearities in the dispersive system \eqref{KG} and the elliptic changes of variables \eqref{Alx11} and \eqref{Alx17} involve the Riesz transform. It is useful to note that our main spaces are stable with respect to the action of singular integrals. More precisely, for integers $n\geq 1$ let
\begin{equation}\label{symb}
\mathcal{S}^{n}:=\{q:\mathbb{R}^3\to\mathbb{C}:\|q\|_{\mathcal{S}^{n}}:=\sup_{\xi\in\mathbb{R}^3\setminus\{0\}}\sup_{|\rho|\leq n}|\xi|^{|\rho|}|D^\rho_\xi q(\xi)|<\infty\},
\end{equation}
denote classes of symbols satisfying differential inequalities of the H\"{o}rmander--Michlin type.

\begin{lemma}\label{tech3}
Assume that $\widehat{Q f}(\xi)=q(\xi)\cdot \widehat{f}(\xi)$ for some $q\in\mathcal{S}^{10}$. Then
\begin{equation}\label{compat}
\begin{split}
\|Qf\|_{Z_1^\sigma}&\lesssim \|f\|_{Z_1^\sigma},\qquad\text{ for any }\sigma\in\{e,b\} \text{ and }f\in Z_1^\sigma,\\
\|(1+|x|^2)^{1/4}Qf\|_{L^2}&\lesssim \|(1+|x|^2)^{1/4}f\|_{L^2}.
\end{split}
\end{equation}
\end{lemma}

See \cite[Lemma 5.1]{IoPa2} for a similar proof.

\subsubsection{The phase functions} We collect now several properties of the phase functions $\Phi=\Phi_{\sigma\mu\nu}$. In this subsection we assume that $\sigma,\mu,\nu\in\{e,b,-e,-b\}$ (so $\mu\neq 0$, $\nu\neq 0$). We start with a suitable description of the geometry of resonant sets. See \cite[Proposition 8.2 and Remark 8.4]{DeIoPa} for proofs; the arguments provided in \cite{DeIoPa} are in two dimensions, but they extend with no difficulty to three dimensions.

\begin{proposition}(Structure of resonance sets)\label{spaceres} The following claims hold:

(i) If either $\nu+\mu=0$ or $\max(|\xi|,|\eta|,|\xi-\eta|)\geq 2^{\D_0}$ or $\min(|\xi|,|\eta|,|\xi-\eta|)\leq 2^{-\D_0}$ then
\begin{equation}\label{res00}|\Phi(\xi,\eta)|\gtrsim (1+|\xi|+|\eta|)^{-1}\quad\mathrm{or}\quad|\nabla_{\eta}\Phi(\xi,\eta)|\gtrsim (1+|\xi|+|\eta|)^{-3}.
\end{equation}

(ii) If $\nu+\mu\neq 0$, then there exists a function $p=p_{\mu\nu}:\mathbb{R}^{2}\to\mathbb{R}^{2}$ such that $|p(\xi)|\lesssim|\xi|$ and $|p(\xi)|\approx |\xi|$ for small $\xi$, and \[\nabla_{\eta}\Phi(\xi,\eta)=0\quad\Leftrightarrow\quad \eta=p(\xi).\] There is an odd smooth function $p_{+}:\mathbb{R}\to\mathbb{R}$, such that $p(\xi)=p_{+}(|\xi|)\xi/|\xi|$. Moreover
\begin{equation}\label{cas10}
\text{ if }\quad|\eta|+|\xi-\eta|\leq U\in[1,\infty)\quad\text{ and }\quad|\nabla_{\eta}\Phi(\xi,\eta)|\leq\varep\quad\text{ then }\quad|\eta-p(\xi)|\lesssim\varep U^4.
\end{equation}
and, for any $s\in\mathbb{R}$,
\begin{equation}\label{cas10.1}
|D^\alpha p_+(s)|\lesssim_{\alpha} 1,\qquad |p'_+(s)|\gtrsim (1+|s|)^{-3},\qquad |1-p'_+(s)|\gtrsim (1+|s|)^{-3}.
\end{equation}

(iii) If $\nu+\mu\neq 0$, we define $p$ as above and $\Psi(\xi):=\Phi(\xi,p(\xi))$. Then $\Psi$ is a radial function, and there exist two positive constants $\gamma_{1}<\gamma_{2}$, such that $\Psi(\xi)=0$ if and only if either\[\pm(\sigma,\mu,\nu)=(b,e,e)\quad\mathrm{and}\quad |\xi|=\gamma_{1},\] or \[\pm(\sigma,\mu,\nu)\in\{(b,e,b),(b,b,e)\}\quad\mathrm{and}\quad |\xi|=\gamma_{2}.\]
\end{proposition}

\begin{remark}\label{largeres} For $\D_0$ sufficiently large we define the function
\begin{equation}\label{reccl}
\Psi_{\sigma}^{\dagger}(\xi)=2^{\D_0}(1+|\xi|)\inf_{\mu,\nu\in\mathcal{P};\,\nu+\mu\neq 0}\left|\Psi_{\sigma\mu\nu}(\xi)\right|
\end{equation}
as in \eqref{psidag}. We have
\begin{equation}\label{cas2}
\Psi_{\pm b}^{\dagger}(\xi)\approx_d 2^{\D_0}\frac{\min\big(\big||\xi|-\gamma_{1}\big|,\big||\xi|-\gamma_{2}\big|\big)}{1+|\xi|}\qquad\text{ and }\qquad10\leq \Psi_{\pm e}^{\dagger}(\xi)\lesssim 1.
\end{equation}
\end{remark}

Our last lemmas are connected to the application of the Schur's test. See \cite[Lemma 8.7 and Proposition 8.8]{DeIoPa} for the proofs. 

We start with a general upper bound on the size of sublevel sets of functions. 

\begin{lemma}\label{lemma00}
Suppose $L,R,M\in\mathbb{R}$, $M\geq \max(1,L,L/R)$, and $Y:B_R:=\{x\in\mathbb{R}^n:|x|<R\}\to\mathbb{R}$ is a function satisfying $\|\nabla Y\|_{C^{l}(B_R)}\leq M$, for some $l\geq 1$. Then, for any $\epsilon>0$,
\begin{equation}\label{scale1}
\big|\big\{x\in B_R:|Y(x)|\leq\epsilon\text{ and }\sum_{|\alpha|\leq l}|\partial_{x}^{\alpha}Y(x)|\geq L\big\}\big|\lesssim R^{n}ML^{-1-1/l}\epsilon^{1/l}.
\end{equation} 
Moreover, if $n=l=1$, $K$ is a union of at most $A$ intervals, and $|Y'(x)|\geq L$ on K, then
\begin{equation}\label{scale2}\left|\{x\in K:|Y(x)|\leq\epsilon\}\right|\lesssim AL^{-1}\epsilon.\end{equation}
\end{lemma}

As a consequence, we have precise bounds on the sublevel sets of our phase functions:

\begin{lemma}\label{Shur2Lem} Assume that $R\geq 1$, $k\geq 0$, and $\epsilon\leq 1/2$. Let
\begin{equation*}E=\{(\xi,\eta):\max(|\xi|,|\eta|)\leq 2^{k},\,|\xi-\eta|\leq R,|\Phi(\xi,\eta)|\leq 2^{-k}\epsilon\}.
\end{equation*}
Then
\begin{equation}\label{cas4} \sup_{\xi}\int_{\mathbb{R}^{3}}\mathbf{1}_{E}(\xi,\eta)\,d\eta+\sup_{\eta}\int_{\mathbb{R}^{3}}\mathbf{1}_{E}(\xi,\eta)\,d\xi\lesssim 2^{5k}R^3\epsilon\log(1/\epsilon).
\end{equation}
\end{lemma}

\subsubsection{Linear Estimates}

We prove now several linear estimates. Given a function $f$, $(k,j)\in\mathcal{J}$, and $n\in\{0,\ldots,j+1\}$ (recall the notation in subsection \ref{NotDef}) we define
\begin{equation}\label{Alx100}
f_{j,k}:=P_{[k-2,k+2]}Q_{jk}f,\qquad \widehat{f_{j,k,n}}(\xi):=\varphi_{-n}^{[-j-1,0]}(\Psi^\dagger_\sigma(\xi))\widehat{f_{j,k}}(\xi).
\end{equation}
Notice that $f_{j,k,n}$ is nontrivial only if $n=0$ or ($n\geq 1$, $\sigma=b$, and $2^k\approx 1$). Moreover,
\begin{equation}\label{Alx100.5}
f_{j,k}=\sum_{n\in[0,j+1]}f_{j,k,n},\qquad P_kf=\sum_{j\geq \max(-k,0)}f_{j,k},\qquad f=\sum_{k\in\mathbb{Z}}P_kf.
\end{equation}

\begin{lemma}\label{LinEstLem}
(i) Assume $\sigma\in\{e,b\}$ and
\begin{equation}\label{Zs}
\Vert f\Vert_{Z_1^\sigma}\leq 1.
\end{equation}
If $m\geq 0$ and $|t|\in[2^{m}-1,2^{m+1}]$ then
\begin{equation}\label{LinftyBd}
\| e^{-it\Lambda_\sigma}f_{j,k,n}\|_{L^\infty}\lesssim\min\big(2^{3k/2}2^{-(1+\beta)j}2^{-n/2+4\beta n},2^{5k^+/2}2^{-3m/2}2^{(1/2-\beta)j}2^{4\beta n}\big).
\end{equation}
As a consequence, for any $k\in\Z$ one has
\begin{equation}\label{LinftyBd2}
\Vert e^{-it\Lambda_\sigma}P_kf\Vert_{L^\infty}\lesssim 2^{-(1+\beta)m}2^{(1/2-\b)\,k}
2^{2k^{+}}.
\end{equation}

(ii) Assume $\sigma\in\{e,b\}$, $N\geq 10$, and
\begin{equation}\label{Zs2}
\Vert f\Vert_{Z_1^\sigma}+\Vert f\Vert_{\H^{N}}\leq 1.
\end{equation}
Then, for any $(k,j)\in\mathcal{J}$ and $n\in\{0,\ldots,j+1\}$,
\begin{equation}\label{RadL2}
\big\Vert \sup_{\theta\in\mathbb{S}^2}|\widehat{f_{j,k,n}}(r\theta)|\,\big\Vert_{L^2(r^2dr)}+\big\Vert \sup_{\theta\in\mathbb{S}^2}|f_{j,k,n}(r\theta)|\,\big\Vert_{L^2(r^2dr)}\lesssim 2^{-(1-2/N)((1+\b)j-4\beta n)}\,.
\end{equation}

Also, we have
\begin{equation}\label{FLinftybd}
\|\widehat{f_{j,k,n}}\|_{L^\infty}\lesssim 2^{j/2-k}2^{-(1-2/N)((1+\b)j-4\beta n)},
\end{equation}
\begin{equation}\label{FLinftybdDER}
\|D^{\alpha}\widehat{f_{j,k,n}}\|_{L^\infty}\lesssim_{|\alpha|}
2^{|\alpha|j}2^{j/2-k}2^{-(1-2/N)((1+\b)j-4\beta n)}.
\end{equation}

(iii) For any $f\in\H^2$ we have
\begin{equation}\label{Zs3}
\|f_{j,k}\|_{L^\infty}\lesssim 2^{k/2-j}\|f\|_{\H^2}.
\end{equation}
\end{lemma}

\begin{proof} (i)  The hypothesis gives
\begin{equation}\label{Alx101}
\Vert f_{j,k,n}\Vert_{L^2}\lesssim 2^{-(1+\b)j+4\beta n}.
\end{equation}
Using the definition,
\begin{equation*}
\Vert e^{-it\Lambda_\sigma}f_{j,k,n}\Vert_{L^\infty}\lesssim \Vert \widehat{f_{j,k,n}}\Vert_{L^1}\lesssim 2^{3k/2}2^{-(1+\beta)j}2^{-n/2+4\beta n}.
\end{equation*}
On the other hand, if $m\geq 10$ then the usual dispersion estimate gives
\begin{equation*}
\Vert e^{-it\Lambda_\sigma}f_{j,k,n}\Vert_{L^\infty}\lesssim 2^{5k^+/2}2^{-3m/2}\Vert f_{j,k,n}\Vert_{L^1}\lesssim 2^{5k^+/2}2^{-3m/2}2^{(1/2-\beta)j}2^{4\beta n}.
\end{equation*}
The bound \eqref{LinftyBd} follows. The bound \eqref{LinftyBd2} follows also, by summation over $j$ and $n$.

(ii) The hypothesis \eqref{Zs2} shows that $\Vert f_{j,k,n}\Vert_{H^N_{\Omega}}\lesssim 1$, where
\begin{equation*}
\|g\|_{H^m_{\Omega}}:=\sum_{\beta_1+\beta_2+\beta_3\leq m}\|\Omega_1^{\beta_1}\Omega_2^{\beta_2}\Omega_3^{\beta_3}g\|_{L^2}.
\end{equation*}
The first inequality in \eqref{RadL2} follows from the interpolation inequality
\begin{equation*}
\|f\|_{H^p_\Omega}\lesssim \Vert f\Vert_{H^N_{\Omega}}^{p/N}\,\|f\|_{2}^{1-p/N},\qquad p\in[0,N]\cap\mathbb{Z},
\end{equation*}
and the Sobolev embedding (along the spheres $\mathbb{S}^2$)
\begin{equation}\label{Zs4}
\begin{split}
\big\Vert \sup_{\theta\in\mathbb{S}^2}|\widehat{f_{j,k,n}}(r\theta)|\,\big\Vert_{L^2(r^2dr)}
&\lesssim \sum_{m_1+m_2+m_3\leq 2}\Vert \Omega_1^{m_1}\Omega_2^{m_2}\Omega_3^{m_3}\widehat{f_{j,k,n}}\Vert_{L^2}\lesssim \Vert \widehat{f_{j,k,n}}\Vert_{H^2_\Omega}.
\end{split}
\end{equation}
The second inequality follows similarly.

To prove \eqref{FLinftybd}, for $\theta\in\mathbb{S}^2$ fixed we estimate
\begin{equation*}
\|\widehat{f_{j,k,n}}(r\theta)\|_{L^\infty_r}\lesssim 2^{j/2}\|\widehat{f_{j,k,n}}(r\theta)\|_{L^2_r}+2^{-j/2}\|(\partial_r\widehat{f_{j,k,n}})(r\theta)\|_{L^2_r}\lesssim 2^{j/2}2^{-k}\|\widehat{f_{j,k,n}}(r\theta)\|_{L^2(r^2dr)},
\end{equation*}
using the localization of the function $Q_{j,k}f$ in the physical space. The desired bounds \eqref{FLinftybd} follow from \eqref{RadL2}. The bounds in \eqref{FLinftybdDER} follow as well, if we notice that derivatives in $\xi$ corresponds to multiplication by $2^j$ factors, due to space localization.

(iii) We may assume $\|f\|_{\H^2}=1$. Using Sobolev embedding in the spheres, as in \eqref{Zs4},
\begin{equation*}
\big\Vert \sup_{\theta\in\mathbb{S}^2}|Q_{j,k}f(r\theta)|\,\big\Vert_{L^2(r^2dr)}
\lesssim 1.
\end{equation*}
The desired estimate follows in the same way as the bound \eqref{FLinftybd}.
\end{proof}

\section{Energy estimates}\label{EneEst}

In this section we prove our main energy estimates. In the rest of the paper we often use the standard Einstein convention that repeated indices are summed. We work in the physical space and divide the proofs into two parts: a high order estimate for the full system (the $\widetilde{\H}^{N_0}$ norm in \eqref{bootstrap3}), and a weighted estimate only for the vorticity components (the estimate \eqref{bootstrapV3}).

\subsection{The total energy of the system}\label{TotalEnergy}

In this subsection we prove the following:
\begin{proposition}\label{BootstrapEE1}
With the hypothesis in Proposition \ref{bootstrap}, we have, for any $t\in[0,T]$,
\begin{equation}\label{bootstrapimp3}
\|(n(t),v(t),E(t),B(t))\|_{\widetilde{\H}^{N_0}}\leq \overline{C}\bar{\eps}/2.
\end{equation}
\end{proposition}

\begin{proof} Recall the real-valued variables $F,G,Z,W$ defined in \eqref{Alx11},
\begin{equation}\label{Alx11.1}
F=|\nabla|^{-1}\hbox{div}(v),\qquad G=|\nabla|^{-1}\nabla\times v,\qquad Z=|\nabla|^{-1}\hbox{div}(E),\qquad W=|\nabla|^{-1}\nabla\times E,
\end{equation}
and the system \eqref{Alx13} (written now in terms of the variables $F,G,Z,W,B$),\footnote{It is important to write the system in terms of these variables, not the more physical variables $n,v,E,B$, in order to be able to prove energy estimates that include the rotation vector-fields.}
\begin{equation}\label{Alx13.5}
\begin{cases}
\partial_t F+(1+d|\nabla|^2)Z&=-R\cdot(v\cdot \nabla v)-R\cdot (v\times B),\\
\partial_t G+W&=-R\times(v\cdot \nabla v)-R\times (v\times B),\\
\partial_t Z-F&=R\cdot (nv),\\
\partial_t W-G-|\nabla|B&=R\times (nv),\\
\partial_t B+|\nabla|W&=0.
\end{cases}
\end{equation}
Recall that $\hbox{div}(B)=0$ and $n=-|\nabla|Z$.

{\bf{Step 1.}} For $m\in[0,N_0]\cap\mathbb{Z}$ we define the energy functionals $\mathcal{E}_{m}:[0,T]\to\mathbb{R}$,
\begin{equation}\label{entot1}
\begin{split}
\mathcal{E}_{m}(t):=\sum_{\mathcal{L}\in\mathcal{V}_{m}}\int_{\mathbb{R}^3}\big\{d|\L n(t)|^2&+(1+n(t))[|\L F(t)|^2+|\L G(t)|^2]\\
&+|\L Z(t)|^2+|\L W(t)|^2+|\L B(t)|^2\big\}\,dx.
\end{split}
\end{equation}
Notice that the case $m=0$ is similar (but not identical, because of the different cubic correction) to the conserved physical energy in \eqref{EnCons}. Notice also that, for any $t\in[0,T]$,
\begin{equation*}
\mathcal{E}_{N_0}(t)\approx \|(n,F,G,Z,W,B)(t)\|_{\H^{N_0}}^2\approx \|(n,v,E,B)(t)\|_{\widetilde{\H}^{N_0}}^2.
\end{equation*}
In particular, there is a constant $C_1\geq 1$ such that, for any $t\in[0,T]$,
\begin{equation}\label{Alx13.6}
C_1^{-1}\mathcal{E}_{N_0}(t)\leq \|(n,v,E,B)(t)\|_{\widetilde{\H}^{N_0}}^2\leq C_1\mathcal{E}_{N_0}(t).
\end{equation}

We would like to estimate now the energy increment. For $\mathcal{L}\in\mathcal{V}_{N_0}$ let $\mathcal{E}_{\mathcal{L}}$ denote the term in \eqref{entot1} corresponding to the differential operator $\mathcal{L}$. We calculate, using \eqref{Alx13.5},
\begin{equation*}
\begin{split}
\frac{d}{dt}\mathcal{E}_{\mathcal{L}}&=\int_{\mathbb{R}^3}\big\{2d\mathcal{L}n\cdot \mathcal{L}[-|\nabla|F-\nabla\cdot(nv)]-[|\nabla|F+\nabla\cdot(nv)]\cdot [|\L F|^2+|\L G|^2]\\
&+2(1+n)\L F\cdot \mathcal{L} [-(1+d|\nabla|^2)Z+\mathcal{N}_F]+2(1+n)\L G\cdot\mathcal{L}[-W+\mathcal{N}_G]\\
&+2\L Z\cdot\mathcal{L}[F+R\cdot(nv)]+2\L W\cdot \mathcal{L}[G+|\nabla|B+R\times(nv)]-2\L B\cdot\L|\nabla|W\big\}\,dx,
\end{split}
\end{equation*}
where $\mathcal{N}_F$ and $\mathcal{N}_G$ denote the nonlinearities corresponding to the equations for $F$ and $G$ in \eqref{Alx13.5}. Since $\mathcal{L}$ and $|\nabla|$ commute, all the quadratic terms in the expression above cancel, so
\begin{equation}\label{Alx13.8}
\begin{split}
\partial_t\mathcal{E}_{\mathcal{L}}&=\int_{\mathbb{R}^3}\big\{-2d\mathcal{L}n\cdot \mathcal{L}(\nabla\cdot(nv))-[|\nabla|F+\nabla\cdot(nv)]\cdot [|\L F|^2+|\L G|^2]\\
&+2(1+n)\L F\cdot \mathcal{L} \mathcal{N}_F-2n\L F\cdot \mathcal{L} (1+d|\nabla|^2)Z+2(1+n)\L G\cdot\mathcal{L}\mathcal{N}_G-2n\L G\cdot\mathcal{L}W\\
&+2\L Z\cdot\mathcal{L}(R\cdot(nv))+2\L W\cdot \mathcal{L}(R\times(nv))\big\}\,dx.
\end{split}
\end{equation}

{\bf{Step 2.}} We would like to show that, for any $t\in[0,T]$,
\begin{equation}\label{Alx13.7}
|\partial_t\mathcal{E}_{\mathcal{L}}(t)|\lesssim \|(n,v,E,B)(t)\|_{\widetilde{\H}^{N_0}}^2\|(n,v,E,B)(t)\|_{\mathcal{W}^{N_0/2,\infty}}.
\end{equation}
All the terms in \eqref{Alx13.8} are at least cubic, but we also need to avoid potential loss of derivatives. Let $\mathcal{A}_2(t):=\|(n,v,E,B)(t)\|_{\widetilde{\H}^{N_0}}$ and $\mathcal{A}_\infty(t):=\|(n,v,E,B)(t)\|_{\mathcal{W}^{N_0/2,\infty}}$. Notice that
\begin{equation*}
\mathcal{A}_\infty(t)\lesssim \mathcal{A}_2(t)\lesssim \bar{\eps}\qquad\text{ for any }t\in[0,T].
\end{equation*}
Some of the terms in \eqref{Alx13.8} can be estimated easily, using the definitions \eqref{Alx11.1}, i.e.
\begin{equation*}
\begin{split}
\Big|\int_{\mathbb{R}^3}&[|\nabla|F+\nabla\cdot(nv)]\cdot [|\L F|^2+|\L G|^2]\,dx\Big|+\Big|\int_{\mathbb{R}^3}n\L F\cdot \mathcal{L} Z\,dx\Big|+\Big|\int_{\mathbb{R}^3}n\L G\cdot\mathcal{L}W\,dx\Big|\\
&+\Big|\int_{\mathbb{R}^3}\L Z\cdot\mathcal{L}(R\cdot(nv))\,dx\Big|+\Big|\int_{\mathbb{R}^3}\L W\cdot \mathcal{L}(R\times(nv))\,dx\Big|\lesssim \mathcal{A}_2^2\mathcal{A}_\infty,
\end{split}
\end{equation*}
since these terms do not lose derivatives.

For the remaining terms, we extract first the components that could lose derivatives. Clearly
\begin{equation*}
\begin{split}
\big\|\mathcal{L}(\nabla\cdot(nv))-[n\L\partial_j v_j+v_j\partial_j\L n]\big\|_{L^2}&\lesssim \mathcal{A}_2\mathcal{A}_\infty,\\
\big\|\mathcal{L}\mathcal{N}_F+R_j(v_k\partial_k\L v_j)\big\|_{L^2}&\lesssim \mathcal{A}_2\mathcal{A}_\infty,\\
\big\|(\mathcal{L}\mathcal{N}_G)_j+\in_{jab}R_a(v_k\partial_k\L v_b)\big\|_{L^2}&\lesssim \mathcal{A}_2\mathcal{A}_\infty.
\end{split}
\end{equation*}
Using the general bound
\begin{equation}\label{Alx13.11}
\|R_j(f\cdot |\nabla|g)-f\cdot R_j|\nabla|g\|_{L^2}\lesssim \|g\|_{L^2}\big(\sum_{k\in\mathbb{Z}}2^k\|P_kf\|_{L^\infty}\big),
\end{equation}
we can further replace $R_j(v_k\partial_k\L v_j)$ by $v_k\cdot \partial_k\L R_jv_j$ and $\in_{jab}R_a(v_k\partial_k\L v_b)$ by $v_k\cdot\in_{jab}\partial_k\L R_av_b$ at the expense of acceptable errors. For \eqref{Alx13.7} it remains to prove that
\begin{equation}\label{Alx13.9}
|\mathcal{E}''_{\mathcal{L}}(t)|\lesssim \mathcal{A}_2(t)^2\mathcal{A}_\infty(t),
\end{equation}
where
\begin{equation*}
\begin{split}
\mathcal{E}''_{\mathcal{L}}=\int_{\mathbb{R}^3}\big\{&-2d\mathcal{L}n\cdot [n\L\partial_j v_j+v_j\partial_j\L n]-2(1+n)\L F\cdot v_k\cdot \partial_k\L R_jv_j\\
&-2dn\L F\cdot \mathcal{L} |\nabla|^2Z-2(1+n)\L G_j\cdot v_k\cdot\in_{jab}\partial_k\L R_av_b\big\}\,dx.
\end{split}
\end{equation*}

Since $R_jv_j=F$ and $\in_{jab}R_av_b=G_j$ we have
\begin{equation*}
\Big|\int_{\mathbb{R}^3}(1+n)\L F\cdot v_k\cdot \partial_k\L R_jv_j\,dx\Big|+\Big|\int_{\mathbb{R}^3}(1+n)\L G_j\cdot v_k\cdot\in_{jab}\partial_k\L R_av_b\,dx\Big|\lesssim \mathcal{A}_2^2\mathcal{A}_\infty.
\end{equation*}
We also have, using integration by parts
\begin{equation*}
\Big|\int_{\mathbb{R}^3}-2d\mathcal{L}n\cdot v_j\partial_j\L n\,dx\Big|\lesssim \mathcal{A}_2^2\mathcal{A}_\infty.
\end{equation*}
Combining the remaining terms in $\mathcal{E}''_{\mathcal{L}}$ and recalling that $n=-|\nabla|Z$ and $\partial_jv_j=|\nabla|F$, it remains to show that
\begin{equation}\label{Alx13.10}
\begin{split}
\Big|\int_{\mathbb{R}^3}\big\{ -n\mathcal{L}n\cdot \L |\nabla|F+n\L F\cdot \mathcal{L} |\nabla|n\big\}\,dx\Big|\lesssim \mathcal{A}_2^2\mathcal{A}_\infty.
\end{split}
\end{equation}
This follows using again the bound \eqref{Alx13.11} and the identity $-|\nabla|=R_j\partial_j$. The desired bound \eqref{Alx13.7} follows.

{\bf{Step 3.}} Given \eqref{Alx13.6}, we estimate first
\begin{equation*}
\begin{split}
\|(n,v,E,B)(t)\|_{\widetilde{\H}^{N_0}}^2&\leq C_1\mathcal{E}_{N_0}(0)+C_1\int_{0}^t|(\partial_s\mathcal{E}_{N_0})(s)|\,ds\\
&\leq C_1^2\|(n,v,E,B)(0)\|_{\widetilde{\H}^{N_0}}^2+C_1\int_{0}^t|(\partial_s\mathcal{E}_{N_0})(s)|\,ds.
\end{split}
\end{equation*}
Since $\|(n,v,E,B)(0)\|_{\widetilde{\H}^{N_0}}^2\lesssim\bar{\eps}^2$ (see \eqref{bootstrap1}), using also \eqref{Alx13.7}, for \eqref{bootstrapimp3} it suffices to show that
\begin{equation}\label{Alx14.1}
\int_0^T\|(n,v,E,B)(t)\|_{\mathcal{W}^{N_0/2,\infty}}\,dt\lesssim \overline{\eps}.
\end{equation}

Using \eqref{Alx17} we have
\begin{equation*}
\|(n,v,E,B)(t)\|_{\mathcal{W}^{N_0/2,\infty}}\lesssim \sum_{k\in\mathbb{Z},\,\mathcal{L}\in\mathcal{V}_{N_0/2}}\big\{\|P_k\mathcal{L}U_e(t)\|_{L^\infty}+\|P_k\mathcal{L}U_b(t)\|_{L^\infty}+\|P_k\mathcal{L}Y(t)\|_{L^\infty}\big\}.
\end{equation*}
Recall that $U_e(t)=e^{-it\Lambda_e}V_e(t)$, $U_b(t)=e^{-it\Lambda_b}V_b(t)$, and $\|(V_e(t),V_b(t))\|_{Z}\lesssim\overline{\eps}$, see \eqref{bootstrap2}. The $L^\infty$ estimates \eqref{LinftyBd2} show that, for any $t\in[0,T]$,
\begin{equation*}
\sum_{k\in\mathbb{Z},\,\mathcal{L}\in\mathcal{V}_{N_0/2}}\big\{\|P_k\mathcal{L}U_e(t)\|_{L^\infty}+\|P_k\mathcal{L}U_b(t)\|_{L^\infty}\big\}\lesssim \bar{\eps}(1+t)^{-1-\beta}.
\end{equation*}
Moreover, recalling the bootstrap assumption \eqref{bootstrapV2}, for any $t\in[0,T]$,
\begin{equation*}
\sum_{k\in\mathbb{Z},\,\mathcal{L}\in\mathcal{V}_{N_0/2}}\|P_k\mathcal{L}Y(t)\|_{L^\infty}\lesssim \delta_0.
\end{equation*}
The desired inequality \eqref{Alx14.1} follows since $T\leq\overline{\eps}/\delta_0$, which completes the proof.
\end{proof}

\subsection{Control of the vorticity energy}\label{vortEn} In this subsection we prove the following:

\begin{proposition}\label{BootstrapEE2}
With the hypothesis in Proposition \ref{bootstrap}, we have, for any $t\in[0,T]$,
\begin{equation}\label{bootstrapimp3.5}
\|(1+|x|^2)^{1/4}Y(t)\|_{\H^{N_1}}\leq \overline{C}\delta_0/2.
\end{equation}
\end{proposition}

\begin{proof}
We define vorticity energy functionals
\begin{equation}\label{env1}
\mathcal{E}^Y_{N_1}(t):=\sum_{\mathcal{L}\in\mathcal{V}_{N_1}}\mathcal{E}^Y_{\L}(t),\qquad \mathcal{E}^Y_{\L}(t):=\int_{\mathbb{R}^3} (1+|x|^2)^{1/2}|\L Y(x,t)|^2\,dx.
\end{equation}
Notice that there is a constant $C_2\geq 1$ such that, for any $t\in[0,T]$,
\begin{equation}\label{Alx14.2}
C_2^{-1}\mathcal{E}^Y_{N_1}(t)\leq\|(1+|x|^2)^{1/4}Y(t)\|^2_{\mathcal{H}^{N_1}}\leq C_2\mathcal{E}^Y_{N_1}(t).
\end{equation}

To prove the proposition we need to estimate the increment of the vorticity energy. More precisely, we would like to show that
\begin{equation}\label{Alx14.3}
\big|\partial_t\mathcal{E}^Y_{\L}(t)\big|\lesssim \delta_0^3+\overline{\eps}(1+t)^{-1-\beta}\delta_0^2.
\end{equation}
Indeed, assuming this, we could estimate, for any $t\in[0,T]$,
\begin{equation*}
\begin{split}
\|(1+|x|^2)^{1/4}Y(t)\|^2_{\H^{N_1}}&\leq C_2\mathcal{E}^Y_{N_1}(0)+C_2\int_0^T\big|\partial_t\mathcal{E}^Y_{N_1}(t)\big|\,dt\\
&\leq C_2^2\|(1+|x|^2)^{1/4}Y(0)\|^2_{\H^{N_1}}+C'\int_0^T(\delta_0^3+\overline{\eps}(1+t)^{-1-\beta}\delta_0^2)\,dt\\
&\leq C_2^2\delta_0^2+C''\overline{\eps}\delta_0^2,
\end{split}
\end{equation*}
where we have used the assumptions \eqref{bootstrapV1} and $T\leq\overline{\eps}/\delta_0$. The desired conclusion \eqref{bootstrapimp3.5} follows, provided that $C_2\ll \overline{C}\ll\overline{\eps}^{\,-1/10}$.

To prove \eqref{Alx14.3}, using the last equation in \eqref{KG} we calculate
\begin{equation*}
\partial_t\mathcal{E}^Y_{\L}=\int_{\mathbb{R}^3} 2(1+|x|^2)^{1/2}\L Y\cdot\L[\nabla\times(v\times Y)] \,dx.
\end{equation*}
Since $\mathrm{div}(Y)=0$ we calculate
\begin{equation*}
[\nabla\times(v\times Y)]_j=Y_l\partial_lv_j-Y_j\partial_lv_l-v_l\partial_lY_j.
\end{equation*}
Recall also that $v=-R\Im(U_e)+R\times\Lambda_b^{-1}\Im(U_b)-R\times\Lambda_b^{-2}|\nabla|Y$, see \eqref{Alx17}. Therefore, after integration by parts to remove the potential derivative loss coming from the term $v_l\partial_lY_j$, we see that $|\partial_t\mathcal{E}^Y_{\L}|$ is bounded by a sum of integrals of the form
\begin{equation}\label{Alx14.4}
C\int_{\mathbb{R}^3}(1+|x|^2)^{1/2}|\L Y|\cdot |Q_1\mathcal{L}_1^a Y|\cdot\big[|Q_2\mathcal{L}_2^b Y|+|\Lambda_2Q_2\mathcal{L}_2^b U_\sigma|\big]\,dx,
\end{equation}
where $a+b\leq N_1$, $\mathcal{L}_1^a\in\mathcal{V}_a$, $\mathcal{L}_2^b\in\mathcal{V}_b$, $Q_1,Q_2$ are operators defined by $\mathcal{S}^{10}$ symbols as in Lemma \ref{tech3}, and $\sigma\in\{e,b\}$. In view of \eqref{compat}, and using the bound
\begin{equation*}
\big\|(1+|x|^2)^{1/4}\L'Y(t)\big\|_{L^2}\lesssim \delta_0
\end{equation*}
for any $t\in[0,T]$ and $\L'\in\mathcal{V}_{N_1}$ (see \eqref{bootstrapV2} and \eqref{Alx14.2}), the integral in \eqref{Alx14.4} is dominated by
\begin{equation*}
C\delta_0^3+C\delta_0^2\|\Lambda_2Q_2\mathcal{L}_2^b U_\sigma\|_{L^\infty}.
\end{equation*}
The desired bound \eqref{Alx14.3} follows once we notice that, using \eqref{LinftyBd2}
\begin{equation*}
\begin{split}
\|\Lambda_2Q_2\mathcal{L}_2^b U_\mu(t)\|_{L^\infty}&\lesssim \sum_{k\in\mathbb{Z}}2^{k^+}\|P_ke^{-it\Lambda_\sigma}\mathcal{L}_2^b V_\sigma(t)\|_{L^\infty}\\
&\lesssim \sum_{k\in\mathbb{Z}}2^{k^+}(1+t)^{-1-\beta}2^{2k^+}2^{(1/2-\beta)k}\|P_k\mathcal{L}_2^b V_\sigma(t)\|_{Z_1^\sigma}\\
&\lesssim (1+t)^{-1-\beta}\sup_{|\alpha|\leq 4}\|D^\alpha\mathcal{L}_2^b V_\sigma(t)\|_{Z_1^\sigma}.
\end{split}
\end{equation*}
This is bounded by $C\overline{\eps}(1+t)^{-1-\beta}$, in view of the bootstrap assumption \eqref{bootstrap3}. The desired conclusion \eqref{Alx14.3} follows, which completes the proof of the proposition.
\end{proof}

\section{Improved control of the $Z$-norm, I: setup and preliminary estimates}\label{ParT}

In the next three sections we prove the following bootstrap estimate for the $Z$-norm.

\begin{proposition}\label{BootstrapZNorm}
With the hypothesis in Proposition \ref{bootstrap}, we have, for any $t\in[0,T]$,
\begin{equation}\label{bootstrapimp3.7}
\|(V_e(t),V_b(t))\|_{Z}\leq \overline{C}\bar{\eps}/2.
\end{equation}
\end{proposition}

\subsection{The Duhamel formula} The functions $U_e$, $U_b$, $Y$ satisfy the equations, (see \eqref{KG})
\begin{equation}\label{za1}
\begin{split}
(\partial_t+i\Lambda_e) U_e&= \Lambda_e (R\cdot [n v])+i |\nabla| (|v|^2)/2-i R\cdot (v\times Y),\\
(\partial_t+i\Lambda_b) U_b&= R\times [nv]- i \Lambda_b^{-1} R\times (v\times Y),\\
\partial_tY&=\nabla\times(v\times Y).
\end{split}
\end{equation}
We define $V_{\sigma}(t)=e^{it\Lambda_{\sigma}}U_{\sigma}(t)$, $\sigma\in\{e,b\}$, as before. Also, for simplicity of notation, let
\begin{equation}\label{za2}
U_0:=Y,\qquad V_0:=Y,\qquad \Lambda_0:=0.
\end{equation}
Since
\begin{equation}\label{za2.5}
n=-|\nabla|\Lambda_e^{-1}\Re(U_e),\qquad v=-R\Im(U_e)+R\times\Lambda_b^{-1}\Im(U_b)-R\times\Lambda_b^{-2}|\nabla|Y,
\end{equation}
see \eqref{Alx17}, our system \eqref{za1} can be written in the form
\begin{equation}\label{system8}
(\partial_{t}+i\Lambda_{\sigma})U_{\sigma}=\sum_{\mu,\nu\in\mathcal{P}'}\mathcal{N}_{\sigma\mu\nu}(U_{\mu},U_{\nu})
\end{equation}
for $\sigma\in\{e,b,0\}$. Here $\mathcal{P}':=\{e,b,-e,-b,0\}$ and the nonlinearities are defined by
\begin{equation}\label{system9}
\left(\mathcal{F}\mathcal{N_{\sigma\mu\nu}}(f,g)\right)(\xi)=\int_{\mathbb{R}^{3}}\mathfrak{m}_{\sigma\mu\nu}(\xi,\eta)\widehat{f}(\xi-\eta)\widehat{g}(\eta)\,d\eta.
\end{equation}
for suitable multipliers $\mathfrak{m}_{\sigma\mu\nu}$ which are sums of functions of the form $m(\xi)m'(\xi-\eta)m''(\eta)$.
In terms of the functions $V_\sigma$, the Duhamel formula is, in the Fourier space,
\begin{equation}\label{duhamelDER}
(\partial_s\widehat{V_{\sigma}})(\xi,s)=\sum_{\mu,\nu\in\mathcal{P}'}\int_{\mathbb{R}^3}e^{is\Phi_{\sigma\mu\nu}(\xi,\eta)}\mathfrak{m}_{\sigma\mu\nu}(\xi,\eta)\widehat{V_{\mu}}(\xi-\eta,s)\widehat{V_{\nu}}(\eta,s)\,d\eta,
\end{equation}
where
\begin{equation*}
\Phi_{\sigma\mu\nu}(\xi,\eta)=\Lambda_\sigma(\xi)-\Lambda_{\mu}(\xi-\eta)-\Lambda_\nu(\eta),\qquad\mu,\nu\in\mathcal{P}'=\{e,b,-e,-b,0\}.
\end{equation*}
In integral form this gives, for $\sigma\in\{e,b\}$ and $t\in[0,T]$,
\begin{equation}\label{duhamel}\widehat{V_{\sigma}}(\xi,t)=\widehat{V_{\sigma}}(\xi,0)+\sum_{\mu,\nu\in\mathcal{P}'}\int_{0}^{t}\int_{\mathbb{R}^3}e^{is\Phi_{\sigma\mu\nu}(\xi,\eta)}\mathfrak{m}_{\sigma\mu\nu}(\xi,\eta)\widehat{V_{\mu}}(\xi-\eta,s)\widehat{V_{\nu}}(\eta,s)\,d\eta ds.
\end{equation}

A rotation vector-field $\Omega\in\{\Omega_1,\Omega_2,\Omega_3\}$ acts on the Duhamel formula according to
\begin{equation*}
\begin{split}
\Omega_\xi&(\partial_s\widehat{V_{\sigma}})(\xi,s)=\sum_{\mu,\nu\in\mathcal{P}'}\int_{\mathbb{R}^3}(\Omega_\xi+\Omega_\eta)\big[e^{is\Phi_{\sigma\mu\nu}(\xi,\eta)}\mathfrak{m}_{\sigma\mu\nu}(\xi,\eta)\widehat{V_{\mu}}(\xi-\eta,s)\widehat{V_{\nu}}(\eta,s)\big]\,d\eta\\
&=\sum_{\mu,\nu\in\mathcal{P}'}\sum_{a_1+a_2+a_3=1}\int_{\mathbb{R}^3}e^{is\Phi_{\sigma\mu\nu}(\xi,\eta)}(\Omega_\xi+\Omega_\eta)^{a_1}\mathfrak{m}_{\sigma\mu\nu}(\xi,\eta)(\Omega^{a_2}\widehat{V_{\mu}})(\xi-\eta,s)(\Omega^{a_3}\widehat{V_{\nu}})(\eta,s)\,d\eta.
\end{split}
\end{equation*}
We iterate this formula. It follows that for any $\L\in\mathcal{V}_{N_1}$ and $\alpha$ we have
\begin{equation}\label{DuhamelDER2}
\begin{split}
\partial_s\widehat{f^{\alpha,\L}_{\sigma}}(\xi,s)=\sum_{\mu,\nu\in\mathcal{P}'}\sum_{|\alpha_1|+|\alpha_2|=|\alpha|}\sum_{(\L_1,\L_2,\L_3)\in X_{\L}}\int_{\mathbb{R}^3}&e^{is\Phi_{\sigma\mu\nu}(\xi,\eta)}\mathfrak{m}_{\sigma\mu\nu}^{\L_3}(\xi,\eta)\\
&\times\widehat{f^{\alpha_1,\L_1}_{\mu}}(\xi-\eta,s)\widehat{f^{\alpha_2,\L_2}_{\nu}}(\eta,s)\,d\eta,
\end{split}
\end{equation}
where here we set
\begin{equation*}
X_{\L}:=\{(\L_1,\L_2,\L_3)\in \mathcal{V}_{N_1}\,|\,|\L_1|+ |\L_2|+|\L_3|\leq |\L|\,\}\,,
\end{equation*}
with $|\L|$ designating the order of the differential operator $\L$, and
\begin{equation}\label{za3}
f^{\beta,\L}_\theta:=D^\beta\L V_\theta,\qquad \theta\in\mathcal{P}',\,|\beta|\leq 4,\,\L\in\mathcal{V}_{N_1}.
\end{equation}
In integral form this becomes
\begin{equation}\label{duhamel2}
\begin{split}
\widehat{f^{\alpha,\L}_{\sigma}}(\xi,t)=\widehat{f^{\alpha,\L}_{\sigma}}(\xi,0)+\sum_{\mu,\nu\in\mathcal{P}'}&\sum_{|\alpha_1|+|\alpha_2|=|\alpha|}\sum_{(\L_1,\L_2,\L_3)\in X_{\L}}\int_0^t\int_{\mathbb{R}^3}e^{is\Phi_{\sigma\mu\nu}(\xi,\eta)}\\
&\times\mathfrak{m}_{\sigma\mu\nu}^{\L_3}(\xi,\eta)\widehat{f^{\alpha_1,\L_1}_{\mu}}(\xi-\eta,s)\widehat{f^{\alpha_2,\L_2}_{\nu}}(\eta,s)\,d\eta.
\end{split}
\end{equation}
We summarize below some of the properties of the functions $f^{\beta,\L}_\theta$ and $\partial_tf^{\beta,\L}_\theta$:

\begin{proposition}\label{sDeriv} (i) The multipliers $\mathfrak{m}_{\sigma\mu\nu}^{\L}$, $\L\in\V_{N_1}$, are sums of functions of the form
\begin{equation}\label{za6}
(1+|\xi|^2)^{1/2}q(\xi)q'(\xi-\eta)q''(\eta),\qquad \|q\|_{\mathcal{S}^n}+\|q'\|_{\mathcal{S}^n}+\|q''\|_{\mathcal{S}^n}\lesssim_n 1,
\end{equation}
for any $n\geq 1$, see \eqref{symb} for the definition of the symbol spaces $\mathcal{S}^{n}$.

(ii) Assume that $|\alpha|\leq 4$ and $\L\in\V_{N_1}$. Then, with the notation in \eqref{za3},
\begin{equation}\label{za4}
\|f_\mu^{\alpha,\L}(t)\|_{\mathcal{H}^{N_0-1-|\L|-|\alpha|}}+\|f_0^{\alpha,\L}(t)\|_{\mathcal{H}^{N_0-1-|\L|-|\alpha|}}+\sup_{\L'\in\V_{N_1-|\L|},\,|\beta|\leq 4-|\alpha|}\|D^\beta\L'f_\mu^{\alpha,\L}(t)\|_{Z_1^\sigma}\lesssim\bar{\eps},
\end{equation}
for any $t\in[0,T]$ and $\mu\in\{e,b\}$. Moreover, letting $\langle t\rangle:=(1+t)$,
\begin{equation}\label{za5}
\|(1+|x|^2)^{1/4}\cdot P_{\leq k}f_0^{\alpha,\L}(t)\|_{\mathcal{H}^{N_1-|\L|}}\lesssim \delta_02^{|\alpha|k}\lesssim \bar{\eps}\langle t\rangle^{-1}2^{|\alpha|k},\qquad k\in\mathbb{Z}_+.
\end{equation}

(iii) For $k\in\mathbb{Z}$, $\sigma\in\{e,b,0\}$,  $\L\in\mathcal{V}_{N_1}$, $|\alpha|\leq 4$, and $t\in[0,T]$ we have
\beq\label{sdL2cont}
\|P_k(\partial_t f^{\alpha,\L}_\sigma)(t)\|_{L^2}\lesssim
\bar{\eps}\min\big\{2^{3k/2},\,2^{-k^+(N_0-2-|\L|-|\alpha|)}\langle t\rangle^{-1},\,\,2^{-k^+(N_1-2-|\L|-|\alpha|)}\langle t\rangle^{-3/2}\big\}.
\eeq
Moreover
\beq\label{sdL2cont2}
\|P_k(\partial_t f^{\alpha,\L}_0)(t)\|_{L^2}\lesssim \bar{\eps}2^{-k^+(N_1-2-|\L|-|\alpha|)}\langle t\rangle^{-2}.
\eeq
\end{proposition}

\begin{proof}
The bounds on the multipliers $\mathfrak{m}_{\sigma\mu\nu}^{\L}$ follow from the explicit formulas for the nonlinearities in \eqref{za1} and the identities \eqref{za2.5}. The bounds \eqref{za4} follow from the bootstrap assumption \eqref{bootstrap2}, while the bounds \eqref{za5} follow from the bootstrap assumption \eqref{bootstrapV2}.

For part (iii) we use the formula \eqref{DuhamelDER2}. We define the operator $I_{\sigma\mu\nu}=I_{\sigma\mu\nu}^\L$ by
\begin{equation}\label{za9}
\mathcal{F}\big\{I_{\sigma\mu\nu}[f,g]\big\}(\xi):=\int_{\mathbb{R}^3}e^{it\Phi_{\sigma\mu\nu}(\xi,\eta)}\mathfrak{m}_{\sigma\mu\nu}^{\L}(\xi,\eta)\widehat{f}(\xi-\eta)\widehat{g}(\eta)\,d\eta.
\end{equation}
We assume that $t\in[0,T]$ is fixed and sometimes drop it from the notation. For $k\in\mathbb{Z}$ let
\begin{equation}\label{za10}
\mathcal{X}_k:=\{(k_1,k_2)\in\mathbb{Z}^2:\,|\max(k_1,k_2)-k|\leq 6\,\text{ or }(\,\max(k_1,k_2)\geq k+7\,\text{ and }\,|k_1-k_2|\leq 6)\}.
\end{equation}

For simplicity of notation let $f_\mu:=f^{\alpha_1,\L_1}_{\mu}$, $f_\nu:=f^{\alpha_2,\L_2}_{\nu}$,  $|\alpha_1|+|\alpha_2|\leq|\alpha|$, $|\L_1|+|\L_2|\leq|\L|$. We estimate first
\begin{equation*}
\|P_kI_{\sigma\mu\nu}[f_\mu,f_\nu]\|_{L^2}\lesssim 2^{3k/2}\|\mathcal{F}\{I_{\sigma\mu\nu}[f_\mu,f_\nu]\}\|_{L^\infty}\lesssim 2^{3k/2}\|f_\mu\|_{\H^1}\|f_\nu\|_{\H^1}\lesssim \bar{\eps}2^{3k/2}
\end{equation*}
for $k\leq 0$, using \eqref{za4} at the last step. This gives the first estimate in \eqref{sdL2cont}. For the second estimate, we write first, using Lemma \ref{L1easy} and \eqref{za6},
\begin{equation*}
\|P_kI_{\sigma\mu\nu}[f_\mu,f_\nu]\|_{L^2}\lesssim 2^{k^+}\sum_{(k_1,k_2)\in\mathcal{X}_k,\,k_1\leq k_2}\|P_{k_1}e^{-it\Lambda_\mu}f_{\mu}\|_{L^\infty}\|P_{k_2}f_{\nu}\|_{L^2}.
\end{equation*}
Using \eqref{za4} we estimate $\|P_{k_2}f_{\nu}\|_{L^2}\lesssim \bar{\eps}2^{-k_2^+(N_0-1-|\L_2|-|\alpha_2|)}$. Using \eqref{za5} and \eqref{LinftyBd2} we estimate
\begin{equation}\label{za11}
\begin{split}
\|P_{k_1}e^{-it\Lambda_\mu}f_{\mu}\|_{L^\infty}\lesssim \bar{\eps}\langle t\rangle^{-1-\beta}2^{k_1/4}2^{3k_1^+}\cdot 2^{-k_1^+(N_1+4-|\L_1|-|\alpha_1|)},\qquad&\text{ if }\mu\in\{e,b,-e,-b\},\\
\|P_{k_1}e^{-it\Lambda_\mu}f_{\mu}\|_{L^\infty}\lesssim \bar{\eps}\langle t\rangle^{-1}2^{-k_1^+(N_1-|\L_1|-|\alpha_1|)}2^{3k_1/2},\qquad&\text{ if }\mu=0,
\end{split}
\end{equation}
where in the second estimate we used the fact that $\delta_0\lesssim\bar{\eps}(1+t)^{-1}$. Therefore, since $|\L_1|+|\L_2|\leq |\L|$ and $|\alpha_1|+|\alpha_2|\leq|\alpha|$ (the worst case is $|\L_1|=0, |\L_2|=|\L|, |\alpha_1|=0, |\alpha_2|=|\alpha|$),
\begin{equation*}
\begin{split}
\|P_kI_{\sigma\mu\nu}[f_\mu,f_\nu]\|_{L^2}&\lesssim 2^{k^+}\sum_{(k_1,k_2)\in\mathcal{X}_k,\,k_1\leq k_2}\langle t\rangle^{-1}2^{k_1/4}2^{-2k_1^+}\cdot \bar{\eps}2^{-k_2^+(N_0-1-|\L|-|\alpha|)}\\
&\lesssim \bar{\eps}\langle t\rangle^{-1}2^{-k^+(N_0-2-|\L|-|\alpha|)},
\end{split}
\end{equation*}
which gives the second bound in \eqref{sdL2cont}.

To prove the last estimate we may assume that $\langle t\rangle\geq 2^{20k^+}$. If $\mu=\nu=0$ then
\begin{equation*}
\begin{split}
\|P_k&I_{\sigma\mu\nu}[f_\mu,f_\nu]\|_{L^2}\lesssim 2^{k^+}\sum_{(k_1,k_2)\in\mathcal{X}_k,\,k_1\leq k_2}\|P_{k_1}f_{\mu}\|_{L^\infty}\|P_{k_2}f_{\nu}\|_{L^2}\\
&\lesssim 2^{k^+}\sum_{(k_1,k_2)\in\mathcal{X}_k,\,k_1\leq k_2}\bar{\eps}\langle t\rangle^{-1}2^{-k_1^+(N_1-|\L_1|-|\alpha_1|)}2^{3k_1/2}\cdot \bar{\eps}\langle t\rangle^{-1}2^{-k_2^+(N_1-|\L_2|-|\alpha_2|)}\\
&\lesssim \bar{\eps}\langle t\rangle^{-2}2^{-k^+(N_1-2-|\L|-|\alpha|)},
\end{split}
\end{equation*}
using \eqref{za11} and \eqref{za5}. Similarly, if $\mu\neq 0$ and $\nu=0$ then
\begin{equation*}
\|P_kI_{\sigma\mu\nu}[f_\mu,f_\nu]\|_{L^2}\lesssim I+II
\end{equation*}
where
\begin{equation*}
\begin{split}
I&:=2^{k^+}\sum_{(k_1,k_2)\in\mathcal{X}_k,\,2^{k_2}\geq \min(\langle t\rangle^{-4},2^{k_1})}\|P_{k_1}f_{\mu}\|_{L^\infty}\|P_{k_2}f_{\nu}\|_{L^2}\\
&\lesssim 2^{k^+}\sum_{(k_1,k_2)\in\mathcal{X}_k\,2^{k_2}\geq \min(\langle t\rangle^{-4},2^{k_1})}\bar{\eps}\langle t\rangle^{-1-\beta}2^{k_1/4}2^{-k_1^+(N_1+1-|\L_1|-|\alpha_1|)}\cdot \bar{\eps}\langle t\rangle^{-1}2^{-k_2^+(N_1-|\L_2|-|\alpha_2|)}\\
&\lesssim \bar{\eps}\langle t\rangle^{-2}2^{-k^+(N_1-2-|\L|-|\alpha|)}
\end{split}
\end{equation*}
and
\begin{equation*}
\begin{split}
II&:=2^{k^+}\sum_{(k_1,k_2)\in\mathcal{X}_k,\,2^{k_2}\leq \min(\langle t\rangle^{-4},2^{k_1})}\|P_{k_1}f_{\mu}\|_{L^2}\|P_{k_2}f_{\nu}\|_{L^\infty}\\
&\lesssim 2^{k^+}\sum_{(k_1,k_2)\in\mathcal{X}_k\,\,2^{k_2}\leq \min(\langle t\rangle^{-4},2^{k_1})}\bar{\eps}2^{-k_1^+(N_0-N_1-5)}\cdot \bar{\eps} \langle t\rangle^{-1}2^{3k_2/2}\\
&\lesssim \bar{\eps}\langle t\rangle^{-2}2^{-k^+(N_1-2-|\L|-|\alpha|)}.
\end{split}
\end{equation*}
These three estimates suffice to prove the desired bound in \eqref{sdL2cont} (since $2^{k^+}\leq \langle t\rangle^{1/20}$), and also the bound \eqref{sdL2cont2} (since either $\mu=0$ or $\nu=0$ when $\sigma=0$, see the last equation in \eqref{za1}).

Finally, assume that $\mu\neq 0$ and $\nu\neq 0$. We decompose
\begin{equation}\label{za15}
\begin{split}
&f_\mu=\sum_{(k_1,j_1)\in\mathcal{J}}f^\mu_{j_1,k_1}=\sum_{(k_1,j_1)\in\mathcal{J}}P_{[k_1-2,k_1+2]}Q_{j_1k_1}f_\mu,\\
&f_\nu=\sum_{(k_2,j_2)\in\mathcal{J}}f^\nu_{j_2,k_2}=\sum_{(k_2,j_2)\in\mathcal{J}}P_{[k_2-2,k_2+2]}Q_{j_2k_2}f_\nu.
\end{split}
\end{equation}
We estimate, using \eqref{LinftyBd} and \eqref{za4},
\begin{equation*}
\begin{split}
\|P_kI_{\sigma\mu\nu}[f_\mu,f_\nu]\|_{L^2}&\lesssim 2^{k^+}\sum_{(k_1,k_2)\in\mathcal{X}_k,\,j_1\leq j_2}\|e^{-it\Lambda_\mu}f_{j_1,k_1}^\mu\|_{L^\infty}\|f_{j_2,k_2}^\nu\|_{L^2}\\
&\lesssim 2^{k^+}\sum_{(k_1,k_2)\in\mathcal{X}_k,\,j_1\leq j_2}\bar{\eps}2^{5k_1^+/2}\langle t\rangle^{-3/2}2^{(1/2+3\beta)j_1}2^{-k_1^+(N_1+4-|\L_1|-|\alpha_1|)}\\
&\qquad\qquad\qquad\qquad\quad\times\bar{\eps}2^{-j_2(1-3\beta)}2^{-k_2^+(N_1+4-|\L_2|-|\alpha_2|)}\\
&\lesssim \bar{\eps}\langle t\rangle^{-3/2}2^{4k^+},
\end{split}
\end{equation*}
using also that in the sum $k_1\geq -j_1\geq-j_2$ and $k_2\geq -j_2$. This finishes the proof of \eqref{sdL2cont}.
\end{proof}

\subsection{The main reduction} We return now to the proof of Proposition \ref{BootstrapZNorm}. We have
\begin{equation*}
\|(V_e(t),V_b(t))\|_Z\lesssim \sup_{\L\in\mathcal{V}_{N_1},\,|\alpha|\leq 4}[\|f^{\alpha,\L}_e\|_{Z^1_e}+\|f^{\alpha,\L}_b\|_{Z^1_b}],
\end{equation*}
in view of Definition \ref{MainZDef}. We use the integral formula \eqref{duhamel2} and decompose the time integral into dyadic pieces. More precisely, given $t\in[0,T]$, we fix a suitable decomposition of the function $\mathbf{1}_{[0,t]}$, i.e. we fix functions $q_0,\ldots,q_{L+1}:\mathbb{R}\to[0,1]$, $|L-\log_2(2+t)|\leq 2$, with the properties
\begin{equation}\label{nh2}
\begin{split}
&\mathrm{supp}\,q_0\subseteq [0,2], \quad \mathrm{supp}\,q_{L+1}\subseteq [t-2,t],\quad\mathrm{supp}\,q_m\subseteq [2^{m-1},2^{m+1}]\text{ for } m\in\{1,\ldots,L\},\\
&\sum_{m=0}^{L+1}q_m(s)=\mathbf{1}_{[0,t]}(s),\qquad q_m\in C^1(\mathbb{R})\text{ and }\int_0^t|q'_m(s)|\,ds\lesssim 1\text{ for }m\in \{1,\ldots,L\}.
\end{split}
\end{equation}
Let $I_m$ denote the support of $q_m$.

For $m\in[0,L+1]$, $\sigma\in\{e,b\}$, $\mu,\nu\in\mathcal{P}'$, $\L\in\mathcal{V}_{N_1}$, we define the bilinear operators $T_m^{\sigma\mu\nu}$ by
\begin{equation}\label{za16}
\mathcal{F}\{T_m^{\sigma\mu\nu}[f,g]\}(\xi):=\int_0^tq_m(s)\int_{\mathbb{R}^3}e^{is\Phi_{\sigma\mu\nu}(\xi,\eta)}\mathfrak{m}_{\sigma\mu\nu}^{\L}(\xi,\eta)\widehat{f}(\xi-\eta,s)\widehat{g}(\eta,s)\,d\eta.
\end{equation}
For Proposition \ref{BootstrapZNorm} it suffices to prove the following:

\begin{proposition}\label{BootstrapZNorm2}
With the hypothesis in Proposition \ref{bootstrap} and the notation above, we have
\begin{equation}\label{za17}
\sum_{k_1,k_2\in\mathbb{Z}}\big\|Q_{jk}T_m^{\sigma\mu\nu}[P_{k_1}f_\mu,P_{k_2}f_\nu]\big\|_{B_j^\sigma}\lesssim \bar{\eps}\,^2 2^{-om},
\end{equation}
for any fixed $t\in[0,T]$, $m\in[0,L+1]$, $(k,j)\in\mathcal{J}$, $\sigma\in\{e,b\}$, $\mu,\nu\in\mathcal{P}'$, $f_\mu=f^{\alpha_1,\L_1}_{\mu}$, $f_\nu=f^{\alpha_2,\L_2}_{\nu}$, $|\L_1|+|\L_2|\leq N_1$, $|\alpha_1|+|\alpha_2|\leq 4$. Here $o:=10^{-8}$ is a small constant.
\end{proposition}

We prove this proposition in the next two sections. We remove first the contribution of very low and very high input frequencies. Then we consider the interactions containing one of the vorticity variables, in which either $\mu=0$ or $\nu=0$ (by symmetry we may assume that $\nu=0$). Finally, in section \ref{DispInter} we consider the purely dispersive interactions, i.e. $\mu,\nu\in\{e,b,-e,-b\}$.

We will often need to localize the phase, in order to be able to integrate by parts in time. For this we define the operators $I_{l,s}^{\sigma\mu\nu}$, $I_{\leq l,s}^{\sigma\mu\nu}$, and $\widetilde{I}_{l,s}^{\sigma\mu\nu}$, $l\in\mathbb{Z}$, by
\begin{equation}\label{vco6}
\begin{split}
\mathcal{F}\big\{I_{l,s}^{\sigma\mu\nu}[f,g]\big\}(\xi)&:=\int_{\mathbb{R}^3}e^{is\Phi_{\sigma\mu\nu}(\xi,\eta)}\varphi_l(\Phi_{\sigma\mu\nu}(\xi,\eta))\mathfrak{m}_{\sigma\mu\nu}^{\L}(\xi,\eta)\widehat{f}(\xi-\eta)\widehat{g}(\eta)\,d\eta,\\
\mathcal{F}\big\{I_{\leq l,s}^{\sigma\mu\nu}[f,g]\big\}(\xi)&:=\int_{\mathbb{R}^3}e^{is\Phi_{\sigma\mu\nu}(\xi,\eta)}\varphi_{\leq l}(\Phi_{\sigma\mu\nu}(\xi,\eta))\mathfrak{m}_{\sigma\mu\nu}^{\L}(\xi,\eta)\widehat{f}(\xi-\eta)\widehat{g}(\eta)\,d\eta,\\
\mathcal{F}\big\{\widetilde{I}_{l,s}^{\sigma\mu\nu}[f,g]\big\}(\xi)&:=\int_{\mathbb{R}^3}e^{is\Phi_{\sigma\mu\nu}(\xi,\eta)}\widetilde{\varphi}_l(\Phi_{\sigma\mu\nu}(\xi,\eta))\mathfrak{m}_{\sigma\mu\nu}^{\L}(\xi,\eta)\widehat{f}(\xi-\eta)\widehat{g}(\eta)\,d\eta,
\end{split}
\end{equation}
where $\widetilde{\varphi}_l(x):= (2^l/x)\varphi_l(x)$. Then we define the operators $T_{m,l}^{\sigma\mu\nu}$, $T_{m,\leq l}^{\sigma\mu\nu}$, $l\in\mathbb{Z}$, by
\begin{equation}\label{vco6.1}
T_{m,l}^{\sigma\mu\nu}[f,g]:=\int_0^tq_m(s)I_{l,s}^{\sigma\mu\nu}[f(s),g(s)]\,ds,\quad T_{m,\leq l}^{\sigma\mu\nu}[f,g]:=\int_0^tq_m(s)I_{\leq l,s}^{\sigma\mu\nu}[f(s),g(s)]\,ds,
\end{equation}
compare with \eqref{za16}. We record the integration by parts identity
\begin{equation}\label{vco6.2}
\begin{split}
T_{m,l}^{\sigma\mu\nu}&[f,g]=i2^{-l}\int_{0}^tq'_m(s)\widetilde{I}_{l,s}^{\sigma\mu\nu}[f(s),g(s)]\,ds\\
&+i2^{-l}\int_{0}^tq_m(s)\widetilde{I}_{l,s}^{\sigma\mu\nu}[(\partial_sf)(s),g(s)]\,ds+i2^{-l}\int_{0}^tq_m(s)\widetilde{I}_{l,s}^{\sigma\mu\nu}[f(s),(\partial_sg)(s)]\,ds.
\end{split}
\end{equation}

\section{Improved control of the $Z$-norm, II: vorticity interactions}\label{Sec:Z1Norm}

We start with a lemma that applies for all $\mu,\nu\in\mathcal{P}'$.

\begin{lemma}\label{Vo1}
(Very large or very small input frequencies) We have
\begin{equation}\label{vco1}
\sum_{\max(k_1,k_2)\geq j/41+\beta m-\D}\big\|Q_{jk}T_m^{\sigma\mu\nu}[P_{k_1}f_\mu,P_{k_2}f_\nu]\big\|_{B_j^\sigma}\lesssim \bar{\eps}\,^22^{-om},
\end{equation}
and
\begin{equation}\label{vco2}
\sum_{\min(k_1,k_2)\leq -(2/3)(m+j)(1+\beta)}\big\|Q_{jk}T_m^{\sigma\mu\nu}[P_{k_1}f_\mu,P_{k_2}f_\nu]\big\|_{B_j^\sigma}\lesssim \bar{\eps}\,^2 2^{-om}.
\end{equation}
\end{lemma}

\begin{proof}
We estimate, using Definition \ref{MainZDef}, Lemma \ref{L1easy}, \eqref{za4}, and \eqref{za11},
\begin{equation*}
\begin{split}
\|Q_{jk}T_m^{\sigma\mu\nu}[P_{k_1}f_\mu,P_{k_2}f_\nu]\big\|_{B_j^\sigma}&\lesssim 2^{k^+}2^{(1+\beta)j}2^m\sup_{s\in I_m}\|e^{-is\Lambda_\mu}P_{k_1}f_\mu(s)\|_{L^\infty}\|P_{k_2}f_\nu(s)\|_{L^2}\\
&\lesssim 2^{k^+}\bar{\eps}\,^22^{(1+\beta)j}2^{\min(k_1,0)/4}2^{-(N_0-N_1-5)k_2^+},
\end{split}
\end{equation*}
if $k_1\leq k_2$. The bound \eqref{vco1} follows by summation over $(k_1,k_2)\in\mathcal{X}_k$ with $k_2\geq k_1$, $k_2\geq j/41+\beta m$. For the second bound we estimate
\begin{equation*}
\begin{split}
\|Q_{jk}T_m^{\sigma\mu\nu}[P_{k_1}f_\mu,P_{k_2}f_\nu]\big\|_{B_j^\sigma}&\lesssim 2^{k^+}2^{(1+\beta)j}2^m\sup_{s\in I_m}2^{3k_1/2}\|P_{k_1}f_\mu(s)\|_{L^2}\|P_{k_2}f_\nu(s)\|_{L^2}\\
&\lesssim 2^{k^+}\bar{\eps}\,^22^{(1+\beta)j}2^m2^{3k_1/2}2^{-4k_2^+},
\end{split}
\end{equation*}
if $k_1\leq k_2$. The bound \eqref{vco2} follows.
\end{proof}

In the rest of the section we prove Proposition \ref{BootstrapZNorm2} when $\nu=0$. For simplicity of notation, in the rest we drop the superscripts $\sigma\mu\nu$, and write simply $T_m$ instead of $T_m^{\sigma\mu\nu}$, $\widetilde{I}_{l,s}$ instead of $\widetilde{I}_{l,s}^{\sigma\mu\nu}$ etc.  We divide the proof into several lemmas, depending on the relative sizes of the main variables. In view of Lemma \ref{Vo1}, we need to consider only $\approx (j+m)^2$ pairs $(k_1,k_2)$; thus it suffices to prove that
\begin{equation}\label{vco3}
\big\|Q_{jk}T_m[P_{k_1}f_\mu,P_{k_2}f_0]\big\|_{B_j^\sigma}\lesssim \bar{\eps}\,^2 2^{-2om-2oj},
\end{equation}
where the pair $(k_1,k_2)$ is fixed and satisfies
\begin{equation}\label{vco4}
k_1,k_2\in[-(2/3)(m+j)(1+\beta),j/41+\beta m-\D].
\end{equation}

\begin{lemma}\label{Vo2}
(Approximate finite speed of propagation) The bound \eqref{vco3} holds provided that
\begin{equation*}
j\geq \max(-k,m)+\D.
\end{equation*}
\end{lemma}

\begin{proof} We define $f^\mu_{j_1,k_1}$ and $f^0_{j_2,k_2}$ as in \eqref{Alx100}. Integration by parts in $\xi$ together with the change of variables $\eta\to\xi-\eta$ show that the contribution is negligible unless $\min(j_1,j_2)\geq 99 j/100$. On the other hand, for any $j_1,j_2$, we can estimate
\begin{equation*}
\big\|Q_{jk}T_m[f_{j_1,k_1}^\mu,f_{j_2,k_2}^0]\big\|_{B_j^\sigma}\lesssim 2^{j(1+\beta)}2^m\sup_{s\in I_m}2^{k^+}\|e^{-is\Lambda_\mu}f_{j_1,k_1}^\mu(s)\|_{L^{\infty}}\|f_{j_2,k_2}^0(s)\|_{L^{2}},
\end{equation*}
using Lemma \ref{L1easy}. Then we estimate $\|f_{j_2,k_2}^0(s)\|_{L^{2}}\lesssim \bar{\eps}2^{-m}2^{-j_2/2}2^{-k_2^+(N_1-|\L_2|-|\alpha_2|)}$ (using \eqref{za5}), and $\|e^{-is\Lambda_\mu}f_{j_1,k_1}^\mu(s)\|_{L^{\infty}}\lesssim \bar{\eps}2^{-j_1}2^{-k_1^+(N_1-|\L_1|-|\alpha_1|)}$ (using \eqref{LinftyBd} if $\mu\neq 0$ and \eqref{Zs3} if $\mu=0$). Therefore
\begin{equation*}
\big\|Q_{jk}T_m[f_{j_1,k_1}^\mu,f_{j_2,k_2}^0]\big\|_{B_j^\sigma}\lesssim \bar{\eps}\,^22^{j(1+\beta)}2^{6\max(k_1^+,k_2^+)}2^{-(3/2)\min(j_1,j_2)}.
\end{equation*}
The desired conclusion \eqref{vco3} follows by summing over pairs $(j_1,j_2)$ with $\min(j_1,j_2)\geq 99 j/100$, and recalling that $\max(k_1^+,k_2^+)\leq j/30$, see \eqref{vco4}.
\end{proof}

\begin{lemma}\label{Vo3}
The bound \eqref{vco3} holds provided that
\begin{equation*}
j\leq \max(-k,m)+\D\qquad\text{ and }\qquad\mu=0.
\end{equation*}
\end{lemma}

\begin{proof} In this case $|\Phi_{\sigma\mu\nu}(\xi,\eta)|=|\Lambda_\sigma(\xi)|\approx 2^{k_+}$ in the support of the integral, so we can integrate by parts in time. Using \eqref{vco6.2}, it suffices to prove that
\begin{equation}\label{vco10}
\begin{split}
2^{-k^+}2^{j(1+\beta)}\big[\|P_k&\widetilde{I}_{l,s}[P_{k_1}f_0(s),P_{k_2}f_0(s)]\|_{L^2}+2^m\|P_k\widetilde{I}_{l,s}[P_{k_1}(\partial_sf_0)(s),P_{k_2}f_0(s)]\|_{L^2}\\
&+2^m\|P_k\widetilde{I}_{l,s}[P_{k_1}f_0(s),P_{k_2}(\partial_sf_0)(s)]\|_{L^2}\big]\lesssim\bar{\eps}\,^2 2^{-2om-2oj},
\end{split}
\end{equation}
for any $s\in I_m$ and $l\in\mathbb{Z}$ with $|l-k^+|\lesssim 1$. Using \eqref{za5} and the last bound in \eqref{sdL2cont}, we have
\begin{equation*}
\|P_k\widetilde{I}_{l,s}[P_{k_1}f_0,P_{k_2}f_0]\|_{L^2}\lesssim 2^{3k/2}\|P_{k_1}f_0\|_{L^2}\|P_{k_2}f_0\|_{L^2}\lesssim 2^{3k/2}\bar{\eps}\,^22^{-2m}2^{4\max(k_1^+,k_2^+)},
\end{equation*}
\begin{equation*}
\|P_k\widetilde{I}_{l,s}[P_{k_1}(\partial_sf_0),P_{k_2}f_0]\|_{L^2}\lesssim 2^{3k/2}\|P_{k_1}(\partial_sf_0)\|_{L^2}\|P_{k_2}f_0\|_{L^2}\lesssim 2^{3k/2}\bar{\eps}\,^22^{-5m/2}2^{6k_1^+}2^{6k_2^+},
\end{equation*}
and similarly
\begin{equation*}
\|P_k\widetilde{I}_{l,s}[P_{k_1}f_0,P_{k_2}(\partial_sf_0)]\|_{L^2}\lesssim 2^{3k/2}\bar{\eps}\,^22^{-5m/2}2^{6k_1^+}2^{6k_2^+}.
\end{equation*}
Therefore, the left-hand side of \eqref{vco10} is bounded by
\begin{equation*}
C2^{-k^+}(2^{m}+2^{-k})^{1+\beta}2^{3k/2}\bar{\eps}\,^22^{-3m/2}2^{6k_1^+}2^{6k_2^+}\lesssim
\begin{cases}
\bar{\eps}\,^22^{\beta m-m/2}2^{13\max(k_1^+,k_2^+)}&\text{ if }m\geq -k,\\
\bar{\eps}\,^22^{k/2-\beta k}2^{13\max(k_1^+,k_2^+)}&\text{ if }m\leq -k.
\end{cases}
\end{equation*}
The desired conclusion \eqref{vco10} follows since $2^{\max(k_1^+,k_2^+)}\lesssim (2^m+2^{-k})^{1/30}2^{\beta m}$, see \eqref{vco4}.
\end{proof}

\begin{lemma}\label{Vo4}
The bound \eqref{vco3} holds provided that
\begin{equation*}
j\leq -k+2\D\qquad\text{ and }\qquad\mu\in\{e,b,-e,-b\}.
\end{equation*}
\end{lemma}

\begin{proof} Clearly $k\leq 2\D$. We estimate first, using \eqref{za4}--\eqref{za5},
\begin{equation*}
\begin{split}
\|Q_{jk}T_m[P_{k_1}f_\mu,P_{k_2}f_0]\big\|_{B_j^\sigma}&\lesssim 2^{(1+\beta)(-k)}2^m 2^{3k/2}\sup_{s\in I_m}\|P_{k_1}f_\mu(s)\|_{L^2}\|P_{k_2}f_0(s)\|_{L^2}\\
&\lesssim 2^{(1/2-\beta)k}\bar{\eps}\,^22^{-5k_1^+}2^{4k_2^+}.
\end{split}
\end{equation*}
This suffices to prove \eqref{vco3} unless
\begin{equation*}
m\geq -100k\qquad\text{ and }\qquad m\geq 100\max(k_1^+,k_2^+).
\end{equation*}
On the other hand, if both these inequalities hold then we estimate the $L^\infty$ norm of the dispersive term using \eqref{LinftyBd2},
\begin{equation*}
\begin{split}
\|Q_{jk}T_m[P_{k_1}f_\mu,P_{k_2}f_0]\big\|_{B_j^\sigma}&\lesssim 2^{(1+\beta)(-k)}2^m \sup_{s\in I_m}\|e^{-is\Lambda_\mu}P_{k_1}f_\mu(s)\|_{L^\infty}\|P_{k_2}f_0(s)\|_{L^2}\\
&\lesssim 2^{(1+\beta)(-k)}\bar{\eps}\,^22^{-(1+\beta)m}2^{7\max(k_1^+,k_2^+)},
\end{split}
\end{equation*}
which suffices to complete the proof of the lemma.
\end{proof}

\begin{lemma}\label{Vo5}
The bound \eqref{vco3} holds provided that
\begin{equation*}
-k+2\D\leq j\leq m+\D\qquad\text{ and }\qquad\mu\in\{e,b,-e,-b\}.
\end{equation*}
\end{lemma}

\begin{proof} Let $\overline{k}:=\max(k_1^+,k_2^+)$ and define $f^\mu_{j_1,k_1}$ and $f^0_{j_2,k_2}$ as in \eqref{Alx100}. We consider three cases:

{\bf{Case 1.}} Assume that
\begin{equation}\label{vco20}
|k_1^+-k_2^+|\leq\D.
\end{equation}
Then we estimate, using \eqref{za4}--\eqref{za5} and the last inequality in \eqref{LinftyBd},
\begin{equation}\label{vco20.5}
\begin{split}
\big\|Q_{jk}&T_m[f_{j_1,k_1}^\mu,f_{j_2,k_2}^0]\big\|_{B_j^\sigma}\lesssim 2^{j(1+\beta)}2^m\sup_{s\in I_m}2^{k^+}\|e^{-is\Lambda_\mu}f_{j_1,k_1}^\mu(s)\|_{L^{\infty}}\|f_{j_2,k_2}^0(s)\|_{L^{2}}\\
&\lesssim 2^{m(1+\beta)}2^m 2^{k^+}\cdot\bar{\eps}\min(2^{-3m/2}2^{(1/2+3\beta)j_1},2^{-(1+\beta)j_1})\bar{\eps}2^{-m}2^{-j_2/2}2^{-4\overline{k}}\\
&\lesssim \bar{\eps}^22^{m/500}2^{-|m-j_1|/2}2^{-j_2/2}2^{-3\overline{k}}.
\end{split}
\end{equation}
The desired conclusion follows for the sum over the pairs $(j_1,j_2)$ with either $|j_1-m|\geq m/100$ or $j_2\geq m/100$.

It remains to consider the pairs $(j_1,j_2)$ with
\begin{equation}\label{vco21}
|j_1-m|\leq m/100\qquad\text{ and }\qquad|j_2|\leq m/100.
\end{equation}
For such pairs we need additional localization in modulation. Recall the notation in \eqref{vco6}--\eqref{vco6.1}. With $l_0:=-m/10$ we estimate
\begin{equation*}
\begin{split}
\big\|Q_{jk}&T_{m,\leq l_0}[f_{j_1,k_1}^\mu,f_{j_2,k_2}^0]\big\|_{B_j^\sigma}\lesssim 2^{m(1+\beta)}2^m\sup_{s\in I_m}\big\|P_kI_{\leq l_0,s}[f_{j_1,k_1}^\mu(s),f_{j_2,k_2}^0(s)]\big\|_{L^2}\\
&\lesssim 2^{m(1+\beta)}2^m 2^{k^+}\sup_{s\in I_m}\Big\|\int_{\mathbb{R}^3}\varphi_{\leq l_0}(\Phi_{\sigma\mu\nu}(\xi,\xi-\eta))\varphi_k(\xi)\,|\widehat{f_{j_1,k_1}^\mu}(\eta,s)|\,|\widehat{f_{j_2,k_2}^0}(\xi-\eta,s)|\,d\eta\Big\|_{L^2_\xi}.
\end{split}
\end{equation*}
We estimate the $L^2$ norm in the expression above using Schur's test. Moreover
\begin{equation}\label{lin}
\|\widehat{f_{j_2,k_2}^0}(s)\|_{L^\infty}\lesssim \|f_{j_2,k_2}^0(s)\|_{L^1}\lesssim 2^{3j_2/2}\|f_{j_2,k_2}^0(s)\|_{L^2}\lesssim 2^{j_2}2^{-\overline{k}(N_1-|\L_2|-|\alpha_2|)}\bar{\eps}2^{-m},
\end{equation}
using \eqref{za5} for the last estimate. Applying now Lemma \ref{Shur2Lem} and \eqref{lin} we get
\begin{equation*}
\begin{split}
\sup_{\xi\in\mathbb{R}^3}\int_{\mathbb{R}^3}\varphi_{\leq l_0}(\Phi_{\sigma\mu\nu}(\xi,\xi-\eta))\varphi_{[k-4,k+4]}(\xi)\varphi_{[k_1-4,k_1+4]}(\eta)|\widehat{f_{j_2,k_2}^0}(\xi-\eta,s)|\,d\eta\lesssim \\
+\sup_{\eta\in\mathbb{R}^3}\int_{\mathbb{R}^3}\varphi_{\leq l_0}(\Phi_{\sigma\mu\nu}(\xi,\xi-\eta))\varphi_{[k-4,k+4]}(\xi)\varphi_{[k_1-4,k_1+4]}(\eta)|\widehat{f_{j_2,k_2}^0}(\xi-\eta,s)|\,d\xi\\
\lesssim 2^{j_2}2^{-\overline{k}(N_1-|\L_2|-|\alpha_2|-9)}2^{l_0}\bar{\eps}2^{-m} (1+m).
\end{split}
\end{equation*}
Using Definition \ref{MainZDef} and \eqref{za5},
\begin{equation*}
\|\widehat{f_{j_1,k_1}^\mu}(s)\|_{L^2}\lesssim 2^{-j_1(1-3\beta)}\|f_\mu(s)\|_{Z_1^\mu}\lesssim \bar{\eps}2^{-j_1+3\beta j_1}2^{-\overline{k}(N_1-|\L_1|-|\alpha_1|)}.
\end{equation*}
Therefore, by Schur's lemma and recalling that $l_0=-m/10$,
\begin{equation}\label{vc022}
\big\|Q_{jk}T_{m,\leq l_0}[f_{j_1,k_1}^\mu,f_{j_2,k_2}^0]\big\|_{B_j^\sigma}\lesssim \bar{\eps}^22^{m(1+2\beta)}2^{-j_1+3\beta j_1}2^{j_2}2^{l_0}2^{-2\overline{k}}.
\end{equation}
Notice that this suffices to control the contribution of the pairs $(j_1,j_2)$ as in \eqref{vco21}.

It remains to control the control the contribution of the larger modulations $l\geq l_0+1$. For this we integrate by parts in time, as in Lemma \ref{Vo3}. Using \eqref{vco6.2} we bound
\begin{equation*}
\begin{split}
\big\|Q_{jk}T_{m,l}&[f_{j_1,k_1}^\mu,f_{j_2,k_2}^0]\big\|_{B_j^\sigma}\lesssim 2^{m(1+\beta)}2^{-l}\sup_{s\in I_m}\big\{\|P_k\widetilde{I}_{l,s}[f_{j_1,k_1}^\mu(s),f_{j_2,k_2}^0(s)]\|_{L^2}\\
&+2^m\|P_k\widetilde{I}_{l,s}[(\partial_sf_{j_1,k_1}^\mu)(s),f_{j_2,k_2}^0(s)]\|_{L^2}+2^m\|P_k\widetilde{I}_{l,s}[f_{j_1,k_1}^\mu(s),(\partial_sf_{j_2,k_2}^0)(s)]\|_{L^2}\big\}\\
&\lesssim 2^{m(1+\beta)}2^{-l}2^{3k^+}\sup_{s\in I_m}\big\{\|f_{j_1,k_1}^\mu(s)\|_{L^2}\|f_{j_2,k_2}^0(s)\|_{L^2}\\
&+2^m\|(\partial_sf_{j_1,k_1}^\mu)(s)\|_{L^2}\|f_{j_2,k_2}^0(s)\|_{L^2}+2^m\|f_{j_1,k_1}^\mu(s)\|_{L^2}\|(\partial_sf_{j_2,k_2}^0)(s)\|_{L^2}\big\}.
\end{split}
\end{equation*}
Using now \eqref{za4}--\eqref{sdL2cont} we can estimate
\begin{equation*}
\begin{split}
\|f_{j_1,k_1}^\mu(s)\|_{L^2}\|f_{j_2,k_2}^0(s)\|_{L^2}+2^m\|f_{j_1,k_1}^\mu(s)\|_{L^2}\|(\partial_sf_{j_2,k_2}^0)(s)\|_{L^2}&\lesssim \bar{\eps}^22^{-j_1(1-3\beta)}2^{-m/4}2^{-4\overline{k}},\\
2^m\|(\partial_sf_{j_1,k_1}^\mu)(s)\|_{L^2}\|f_{j_2,k_2}^0(s)\|_{L^2}&\lesssim \bar{\eps}^22^{-5m/4}2^{-4\overline{k}}.
\end{split}
\end{equation*}
Therefore, for $j_1\geq m-m/100$ as in \eqref{vco21},
\begin{equation}\label{vco30}
\big\|Q_{jk}T_{m,l}[f_{j_1,k_1}^\mu,f_{j_2,k_2}^0]\big\|_{B_j^\sigma}\lesssim 2^{m(1+\beta)}2^{-l}\cdot \bar{\eps}^22^{-5m/4}2^{m/50}.
\end{equation}
The desired bound \eqref{vco3} follows by combining \eqref{vco20.5}, \eqref{vc022}, and \eqref{vco30}.

{\bf{Case 2.}} Assume now that
\begin{equation}\label{vco40}
k_2^+\geq k_1^++\D.
\end{equation}
In this case $k_2\geq \D$, $|k-k_2|\leq 4$, and $|\Phi_{\sigma\mu\nu}(\xi,\eta)|=|\Lambda_\sigma(\xi)-\Lambda_\mu(\xi-\eta)|\approx 2^k$ in the support in the integral. We are therefore in the case when the modulation is large, so we can integrate by parts in time. As before, using \eqref{vco6.2} we bound, for $|l-k|\leq \D$
\begin{equation*}
\begin{split}
\big\|Q_{jk}T_{m,l}&[P_{k_1}f_\mu,P_{k_2}f_0]\big\|_{B_j^\sigma}\lesssim 2^{m(1+\beta)}2^{-k}\sup_{s\in I_m}\big\{\|P_k\widetilde{I}_{l,s}[P_{k_1}f_\mu(s),P_{k_2}f_0(s)]\|_{L^2}\\
&+2^m\|P_k\widetilde{I}_{l,s}[P_{k_1}(\partial_sf_\mu)(s),P_{k_2}f_0(s)]\|_{L^2}+2^m\|P_k\widetilde{I}_{l,s}[P_{k_1}f_\mu(s),P_{k_2}(\partial_sf_0)(s)]\|_{L^2}\big\}.
\end{split}
\end{equation*}
Using \eqref{za4}--\eqref{sdL2cont} and \eqref{LinftyBd2}, we estimate
\begin{equation*}
\|P_k\widetilde{I}_{l,s}[P_{k_1}f_\mu(s),P_{k_2}f_0(s)]\|_{L^2}\lesssim 2^{k^+}\|e^{-is\Lambda_\mu}P_{k_1}f_\mu(s)\|_{L^\infty}\|P_{k_2}f_0(s)\|_{L^2}\lesssim \bar{\eps}^22^{-(2+\beta)m}2^{6k_2},
\end{equation*}
and similarly
\begin{equation*}
\begin{split}
&2^m\|P_k\widetilde{I}_{l,s}[P_{k_1}(\partial_sf_\mu)(s),P_{k_2}f_0(s)]\|_{L^2}\lesssim \bar{\eps}^22^{-(4/3+\beta)m}2^{6k_2},\\
&2^m\|P_k\widetilde{I}_{l,s}[P_{k_1}f_\mu(s),P_{k_2}(\partial_sf_0)(s)]\|_{L^2}\lesssim \bar{\eps}^22^{-(4/3+\beta)m}2^{6k_2}.
\end{split}
\end{equation*}
The desired conclusion follows in this case once we recall that $k_2\leq m/20$, see \eqref{vco4}.

{\bf{Case 3.}} Finally, assume that
\begin{equation}\label{vco50}
k_1^+\geq k_2^++\D.
\end{equation}
In this case $k_1\geq \D$, $|k-k_1|\leq 4$. We use the same argument as in {\bf{Case 1}}. As in the proof of \eqref{vco20.5}, and using also that $n=0$ in this case,
\begin{equation}\label{vco51}
\big\|Q_{jk}T_m[f_{j_1,k_1}^\mu,f_{j_2,k_2}^0]\big\|_{B_j^\sigma}\lesssim \bar{\eps}^22^{-(1/2-\beta)|m-j_1|}2^{-j_2/2}2^{4k}.
\end{equation}
This suffices to control the contribution of the pairs $(j_1,j_2)$ with $(1-\beta)|m-j_1|+j_2\geq 8k+\beta m$.

On the other hand, if
\begin{equation}\label{vco52}
(1-\beta)|m-j_1|+j_2\leq 8k+\beta m,
\end{equation}
then we decompose dyadically in modulation. The contribution of low modulations $|\Phi_{\sigma\mu\nu}|\leq 2^{l_0}$ can be estimated using Schur's lemma. As in the proof of \eqref{vc022}, we can estimate
\begin{equation}\label{vco53}
\big\|Q_{jk}T_{m,\leq l_0}[f_{j_1,k_1}^\mu,f_{j_2,k_2}^0]\big\|_{B_j^\sigma}\lesssim \bar{\eps}^22^{(1+\beta)m}2^{-(1+\beta)j_1}2^{j_2}2^{l_0}.
\end{equation}
Notice that this suffices to control the contribution of the pairs $(j_1,j_2)$ as in \eqref{vco52} if
\begin{equation}\label{vco54}
-l_0:=(1+\beta)|m-j_1|+j_2+\beta m.
\end{equation}

On the other hand, for $l\geq l_0$ we integrate by parts in time and estimate, as in \eqref{vco30},
\begin{equation*}
\begin{split}
\big\|Q_{jk}T_{m,l}&[f_{j_1,k_1}^\mu,f_{j_2,k_2}^0]\big\|_{B_j^\sigma}\lesssim 2^{m(1+\beta)}2^{-l}\sup_{s\in I_m}\big\{\|P_k\widetilde{I}_{l,s}[f_{j_1,k_1}^\mu(s),f_{j_2,k_2}^0(s)]\|_{L^2}\\
&+2^m\|P_k\widetilde{I}_{l,s}[(\partial_sf_{j_1,k_1}^\mu)(s),f_{j_2,k_2}^0(s)]\|_{L^2}+2^m\|P_k\widetilde{I}_{l,s}[f_{j_1,k_1}^\mu(s),(\partial_sf_{j_2,k_2}^0)(s)]\|_{L^2}\big\}\\
&\lesssim 2^{m(1+\beta)}2^{-l}\big\{\bar{\eps}^22^{4k}2^{-m(2+\beta)}+\bar{\eps}^22^{-15k}2^{-(1+6\beta)m}+\bar{\eps}^22^{4k}2^{-m(2+\beta)}\big\},
\end{split}
\end{equation*}
where in the last line we used Lemma \ref{PhiLocLem}, the bounds \eqref{sdL2cont2} and \eqref{LinftyBd2}, and the bound
\begin{equation*}
\|\partial_sf_{j_1,k_1}^\mu(s)\|_{L^2}\lesssim \bar{\eps}2^{-k_1^+(N_0-3-|\L_1|-|\alpha_1|)}2^{-m(1+6\beta)},
\end{equation*}
which is obtained by interpolation from the last two bounds in \eqref{sdL2cont}. Therefore
\begin{equation}\label{vco60}
\sum_{-l\leq -l_0}\big\|Q_{jk}T_{m,l}[f_{j_1,k_1}^\mu,f_{j_2,k_2}^0]\big\|_{B_j^\sigma}\lesssim \bar{\eps}^22^{-\beta m}+\bar{\eps}^22^{14k}2^{-m+6\beta m},
\end{equation}
recalling that $-l_0\leq 9k+3\beta m$, see \eqref{vco52} and \eqref{vco54}. The desired conclusion follows from \eqref{vco51}, \eqref{vco53}, and \eqref{vco60}. This completes the proof of the lemma.
\end{proof}

\section{Improved control of the $Z$-norm, III: dispersive interactions}\label{DispInter}

In this section we prove Proposition \ref{BootstrapZNorm2} when $\mu,\nu\in\{e,b,-e,-b\}$. In view of Lemma \ref{Vo1} it suffices to prove that
\begin{equation}\label{gmo1}
\big\|Q_{jk}T_m^{\sigma\mu\nu}[P_{k_1}f_\mu,P_{k_2}f_\nu]\big\|_{B_j^\sigma}\lesssim \bar{\eps}\,^2 2^{-2om-2oj},
\end{equation}
where the pair $(k_1,k_2)$ is fixed and satisfies
\begin{equation}\label{gmo2}
k_1,k_2\in[-(2/3)(m+j)(1+\beta),j/41+\beta m-\D].
\end{equation}

The proof we present here is similar to the proof in \cite[Sections 6,7]{DeIoPa}. It is simpler, however, because we work here in 3 dimensions, as opposed to 2 dimensions, and this leads to more favorable dispersion and decay properties of the solutions. For the sake of completeness we provide all the details in the rest of this section.

As in the previous section, we drop the superscripts $\sigma\mu\nu$ and consider several cases. In many estimates below we use the basic bounds on the functions $f_\mu=f_\mu^{\al_1,\L_1}$ and $f_\nu=f_\nu^{\al_2,\L_2}$
\begin{equation}\label{gmo2.1}
\sup_{|\beta|\leq N_1+4-|\L|-|\alpha|}\|D^{\beta}f_\gamma^{\alpha,\L}(t)\|_{Z_1^\gamma}+\|f_\gamma^{\alpha,\L}(t)\|_{\mathcal{H}^{N_0-1-|\L|-|\alpha|}}\lesssim\bar{\eps},
\end{equation}
and, for any $k\in\mathbb{Z}$,
\beq\label{gmo2.2}
\|P_k(\partial_t f^{\alpha,\L}_\gamma)(t)\|_{L^2}\lesssim
\bar{\eps}\min\big\{2^{3k/2},\,2^{-k^+(N_0-2-|\L|-|\alpha|)}\langle t\rangle^{-1},\,\,2^{-k^+(N_1-2-|\L|-|\alpha|)}\langle t\rangle^{-3/2}\big\},
\eeq
see Proposition \ref{sDeriv}, where $(\gamma,\L,\alpha)\in\{(\mu,\L_1,\alpha_1),(\nu,\L_2,\alpha_2)\}$ and $\langle t\rangle=1+t$. Recall also that $|\L_1|+|\L_2|\leq N_1$ and $|\alpha_1|+|\alpha_2|\leq 4$. We will often use the integration by parts formula \eqref{vco6.2}.

We divide the proof into several lemmas, depending on the relative size of the main parameters. As before, we start with the simpler cases and gradually reduce to the main resonant cases in Proposition \ref{mo10}.

\begin{lemma}\label{mo2}
(Approximate finite speed of propagation) The bound \eqref{gmo1} holds provided that \eqref{gmo2} holds and, in addition,
\begin{equation*}
j\geq \max(-k,m)+\D.
\end{equation*}
\end{lemma}

\begin{proof} We define $f^\mu_{j_1,k_1}$ and $f^\nu_{j_2,k_2}$ as in \eqref{Alx100}. As in the proof of Lemma \ref{Vo2}, integration by parts in $\xi$ together with the change of variables $\eta\to\xi-\eta$ show that the contribution is negligible unless $\min(j_1,j_2)\geq j(1-\beta/10)$. Without loss of generality we may assume that $k_1\leq k_2$. For any $j_1,j_2$, we can estimate
\begin{equation}\label{gmo5}
\big\|Q_{jk}T_m[f_{j_1,k_1}^\mu,f_{j_2,k_2}^\nu]\big\|_{B_j^\sigma}\lesssim \bar{\eps}^22^{j(1+\beta)}2^m 2^{3k_2^+}2^{-4k_1^+}2^{-(1+\beta)j_1}2^{-(1+\beta)j_2}.
\end{equation}
Indeed, this follows by an $L^2\times L^\infty$ estimate, using \eqref{gmo2.1}, the first bound in \eqref{LinftyBd}, and Definition \ref{MainZDef} (we decompose in $n$ and place the function with the larger $n$ in $L^\infty$ in order to gain the favorable factor $2^{-n/2+4\beta n}$ in \eqref{LinftyBd}). The desired conclusion follows unless
\begin{equation}\label{gmo6}
k_2^+\geq k_1^++\D\qquad\text{ and }\qquad j_1,j_2\in[j(1-\beta/10),4m/3].
\end{equation}

Assume now that \eqref{gmo6} holds. In particular, $k_2\geq\D$ and $|k-k_2|\leq 4$. We further decompose our operator in modulation. As in Lemma \ref{Vo5}, with $l_0:=-14k-20\beta m$ we estimate
\begin{equation*}
\begin{split}
\big\|Q_{jk}&T_{m,\leq l_0}[f_{j_1,k_1}^\mu,f_{j_2,k_2}^\nu]\big\|_{B_j^\sigma}\lesssim 2^{j(1+\beta)}2^m\sup_{s\in I_m}\big\|P_kI_{\leq l_0,s}[f_{j_1,k_1}^\mu(s),f_{j_2,k_2}^\nu(s)]\big\|_{L^2}\\
&\lesssim 2^{j(1+\beta)}2^m 2^{k^+}\sup_{s\in I_m}\Big\|\int_{\mathbb{R}^3}\varphi_{\leq l_0}(\Phi_{\sigma\mu\nu}(\xi,\eta))\varphi_k(\xi)\,|\widehat{f_{j_1,k_1}^\mu}(\xi-\eta,s)|\,|\widehat{f_{j_2,k_2}^\nu}(\eta,s)|\,d\eta\Big\|_{L^2_\xi}.
\end{split}
\end{equation*}
We estimate the $L^2$ norm in the expression above using Schur's test. Using Lemma \ref{Shur2Lem}, it follows that
\begin{equation}\label{gmo7}
\begin{split}
\big\|Q_{jk}T_{m,\leq l_0}[f_{j_1,k_1}^\mu,f_{j_2,k_2}^\nu]\big\|_{B_j^\sigma}&\lesssim 2^{j(1+\beta)}2^m 2^{k}\sup_{s\in I_m}(2^{10k}2^{l_0+\beta m})^{1/2}\|\widehat{f_{j_1,k_1}^\mu}(s)\|_{L^2}\|\widehat{f_{j_2,k_2}^\nu}(s)\|_{L^2}\\
&\lesssim \bar{\eps}^22^{j(1+\beta)}2^m\cdot 2^{-8\beta m}2^{-j_1(1-3\beta)}2^{-j_2(1+\beta)}.
\end{split}
\end{equation}

On the other hand, for $l\geq l_0+1$ we integrate by parts in time. Using \eqref{vco6.2} we bound
\begin{equation*}
\begin{split}
\big\|&Q_{jk}T_{m,l}[f_{j_1,k_1}^\mu,f_{j_2,k_2}^\nu]\big\|_{B_j^\sigma}\lesssim 2^{m(1+\beta)}2^{-l}\sup_{s\in I_m}\big\{\|P_k\widetilde{I}_{l,s}[f_{j_1,k_1}^\mu(s),f_{j_2,k_2}^\nu(s)]\|_{L^2}\\
&+2^m\|P_k\widetilde{I}_{l,s}[(\partial_sf_{j_1,k_1}^\mu)(s),f_{j_2,k_2}^\nu(s)]\|_{L^2}+2^m\|P_k\widetilde{I}_{l,s}[f_{j_1,k_1}^\mu(s),(\partial_sf_{j_2,k_2}^\nu)(s)]\|_{L^2}\big\}\\
&\lesssim \bar{\eps}^22^{m(1+\beta)}2^{-l}2^{k}\big\{2^{-j_1(1-3\beta)}2^{-j_2(1+\beta)}+2^{-m/2}2^{-j_2(1+\beta)}+2^{-j_1(1-3\beta)}2^{-20k_2}2^{-50\beta m}\big\},
\end{split}
\end{equation*}
where in the last term we estimated $\|\partial_sf_{j_2,k_2}^\nu(s)\|_{L^2}\lesssim \bar{\eps}2^{-m-50\beta m}2^{-30k_2^+}$ (interpolation between the last two bounds in \eqref{gmo2.2}). Therefore, for $j_1,j_2$ as in \eqref{gmo6} and $l_0=-14k-20\beta m$,
\begin{equation}\label{gmo8}
\sum_{l\geq l_0}\big\|Q_{jk}T_{m,l}[f_{j_1,k_1}^\mu,f_{j_2,k_2}^\nu]\big\|_{B_j^\sigma}\lesssim \bar{\eps}^22^{16k}2^{-m/2+30\beta m}+ \bar{\eps}^22^{-\beta m}.
\end{equation}
The desired conclusion follows from \eqref{gmo7} and \eqref{gmo8}.
\end{proof}

\begin{lemma}\label{mo3}
The bound \eqref{gmo1} holds provided that \eqref{gmo2} holds and, in addition,
\begin{equation*}
j\leq -k+2\D.
\end{equation*}
\end{lemma}

\begin{proof} Clearly $k\leq 2\D$, thus $|k_1^+-k_2^+|\leq 3\D$. We define $f^\mu_{j_1,k_1}$ and $f^\nu_{j_2,k_2}$ as before and estimate
\begin{equation*}
\big\|Q_{jk}T_m[f_{j_1,k_1}^\mu,f_{j_2,k_2}^\nu]\big\|_{B_j^\sigma}\lesssim \bar{\eps}^22^{j(1+\beta)}2^m \cdot 2^{-3m/2}2^{-\max(j_1,j_2)(1/2-10\beta)}.
\end{equation*}
Indeed, this follows by estimating the term with the smaller $j$ in $L^\infty$ and using the last bound in \eqref{LinftyBd}, and the term with the larger $j$ in $L^2$ and using the Definition \ref{MainZDef}. The desired conclusion follows unless
\begin{equation}\label{gmo9}
[m+\max(j_1,j_2)](1/2-20\beta)+3\D\leq j\leq -k+2\D.
\end{equation}

Assume now that \eqref{gmo9} holds. In particular $k\leq -\D$. We consider first the high modulations, $l\geq l_0+1$, where $l_0:=-2k_1^+-\D$. Using \eqref{vco6.2} and \eqref{gmo2.1}--\eqref{gmo2.2} we estimate
\begin{equation*}
\begin{split}
\big\|&Q_{jk}T_{m,l}[f_{j_1,k_1}^\mu,f_{j_2,k_2}^\nu]\big\|_{B_j^\sigma}\lesssim 2^{j(1+\beta)}2^{-l}\sup_{s\in I_m}\big\{2^{3k/2}\|f_{j_1,k_1}^\mu(s)\|_{L^2}\|f_{j_2,k_2}^\nu(s)\|_{L^2}\\
&+2^m2^{3k/2}\|\partial_sf_{j_1,k_1}^\mu(s)\|_{L^2}\|f_{j_2,k_2}^\nu(s)\|_{L^2}+2^m2^{3k/2}\|f_{j_1,k_1}^\mu(s)\|_{L^2}\|\partial_sf_{j_2,k_2}^\nu(s)\|_{L^2}\big\}\\
&\lesssim \bar{\eps}^22^{j(1+\beta)}2^{-l}2^{3k/2}2^{-4k_1^+}.
\end{split}
\end{equation*}
Deduce now that
\begin{equation}\label{gmo10}
\sum_{l\geq-2k_1^+-\D+1}\big\|Q_{jk}T_{m,l}[f_{j_1,k_1}^\mu,f_{j_2,k_2}^\nu]\big\|_{B_j^\sigma}\lesssim \bar{\eps}^22^{j(1+\beta)}2^{3k/2}.
\end{equation}
and since $2^{3k/2}2^{j(1+\beta)}\lesssim 2^{k(1/2-\beta)}\lesssim 2^{-m/6-\beta j}$ this takes care of the large modulation case.

To estimate the contribution of small modulations we use first Proposition \ref{spaceres} (i). In particular we examine the integral defining $\mathcal{F}\{P_kT_{m,\leq l_0}[f_{j_1,k_1}^\mu,f_{j_2,k_2}^\nu]\}$ and notice that this integral is nontrivial only when $|\eta|+|\xi-\eta|\leq 2^{\D/2}$. Thus $k_1,k_2\in[-\D,\D]$ and, more importantly $|\nabla_\eta\Phi_{\sigma\mu\nu}(\xi,\eta)|\gtrsim 1$ in the support of the integral. Therefore, using integration by parts in $\eta$ (with Lemma \ref{tech5}),
\begin{equation*}
\|\mathcal{F}P_kT_{m,\leq l_0}[f_{j_1,k_1}^\mu,f_{j_2,k_2}^\nu]\|_{L^\infty}\lesssim \bar{\eps}^22^{-2m}\qquad\text{ if }\qquad \max(j_1,j_2)\leq m-\beta m.
\end{equation*}
On the other hand, if $\max(j_1,j_2)\geq m-\beta m$ then we can estimate directly
\begin{equation*}
\big\|Q_{jk}T_{m\leq l_0}[f_{j_1,k_1}^\mu,f_{j_2,k_2}^\nu]\big\|_{B_j^\sigma}\lesssim 2^{j(1+\beta)}2^m 2^{3k/2}\|f_{j_1,k_1}^\mu\|_{L^2}\|f_{j_2,k_2}^\nu\|_{L^2}\lesssim \bar{\eps}^22^{j(1+\beta)}2^{3k/2}2^{10\beta m}.
\end{equation*}
Therefore, assuming \eqref{gmo9},
\begin{equation}\label{gmo11}
\big\|Q_{jk}T_{m\leq l_0}[f_{j_1,k_1}^\mu,f_{j_2,k_2}^\nu]\big\|_{B_j^\sigma}\lesssim \bar{\eps}^22^{k/4}.
\end{equation}
The desired bound when \eqref{gmo9} is satisfied follows from \eqref{gmo10} and \eqref{gmo11}.
\end{proof}

We can now estimate the contribution of large modulations.

\begin{lemma}\label{mo6}
Assume that \eqref{gmo2} holds and, in addition,
\begin{equation*}
-k+2\D\leq j\leq m+\D.
\end{equation*}
Then
\begin{equation}\label{gmo11.6}
\sum_{l\geq -\D-10\max(k_1^+,k_2^+)-200\beta m}\big\|Q_{jk}T_{m,l}[P_{k_1}f_\mu,P_{k_2}f_\nu]\big\|_{B_j^\sigma}\lesssim \bar{\eps}\,^2 2^{-2om-2oj}.
\end{equation}
\end{lemma}

\begin{proof} Using \eqref{vco6.2}, Lemma \ref{PhiLocLem}, \eqref{LinftyBd}, and \eqref{gmo2.1}--\eqref{gmo2.2} we estimate
\begin{equation*}
\begin{split}
\big\|&Q_{jk}T_{m,l}[P_{k_1}f_\mu,P_{k_2}f_\nu]\big\|_{B_j^\sigma}\lesssim 2^{m(1+\beta)}2^{-l}\sup_{s\in I_m}\big\{\|P_k\widetilde{I}_{l,s}[P_{k_1}f_\mu(s),P_{k_2}f_\nu(s)]\|_{L^2}\\
&+2^m\|P_k\widetilde{I}_{l,s}[P_{k_1}(\partial_sf_\mu)(s),P_{k_2}f_\nu(s)]\|_{L^2}+2^m\|P_k\widetilde{I}_{l,s}[P_{k_1}f_\mu(s),P_{k_2}(\partial_sf_\nu)(s)]\|_{L^2}\big\}\\
&\lesssim \bar{\eps}^22^{-l}2^{8\max(k_1^+,k_2^+)}2^{-m/2+\beta m}.
\end{split}
\end{equation*}
This gives \eqref{gmo11.6}, since $\max(k_1,k_2)\leq m/41+\beta m$.
\end{proof}

\begin{lemma}\label{mo7}
Let $\overline{k}:=\max(k_1^+,k_2^+)$. Assume that \eqref{gmo2} holds and, in addition,
\begin{equation*}
-k+2\D\leq j\leq m+\D.
\end{equation*}
Then
\begin{equation}\label{gmo20}
\big\|Q_{jk}T_{m,\leq -\D-10\overline{k}-200\beta m}[P_{k_1}f_\mu,P_{k_2}f_\nu]\big\|_{B_j^\sigma}\lesssim \bar{\eps}\,^2 2^{-4om}
\end{equation}
provided that
\begin{equation}\label{gmo21}
\mu=-\nu\qquad\text{ or }\qquad\min(k,k_1,k_2)\leq -\D/2\qquad\text{ or }\qquad\max(k,k_1,k_2)\geq \D/2.
\end{equation}
\end{lemma}

\begin{proof}  Using Proposition \ref{spaceres} (i) it follows that $|\nabla_\eta \Phi_{\sigma\mu\nu}(\xi,\eta)|\gtrsim 2^{-3\overline{k}}$ in the support of the integral defining $\mathcal{F}\{P_kT_{m,\leq -\D-10\overline{k}-200\beta m}[P_{k_1}f_\mu,P_{k_2}f_\nu]\}$. We define $f^\mu_{j_1,k_1}$ and $f^\nu_{j_2,k_2}$ as before and notice that the contribution of the components for which $\max(j_1,j_2)\leq m-\beta m-3\overline{k}$ is negligible, using integration by parts in $\eta$ (with Lemma \ref{tech5}).

We consider two cases:

{\bf{Case 1.}} Assume first that
\begin{equation}\label{gmo40}
|k_1^+-k_2^+|\leq\D,\qquad\max(j_1,j_2)\geq m-\beta m-3\overline{k}.
\end{equation}
In this case we do not lose derivatives. Assuming, without loss of generality, that $j_1\leq j_2$ we estimate first
\begin{equation}\label{gmo41}
\begin{split}
\big\|Q_{jk}&T_{m,\leq -\D-10\overline{k}-50\beta m}[f_{j_1,k_1}^\mu,f_{j_2,k_2}^\nu]\big\|_{B_j^\sigma}\\
&\lesssim 2^{j(1+\beta)}2^m 2^{\overline{k}}\big\{\sup_{s\in I_m,\,|t-s|\leq 2^{m/2}}\|e^{-it\Lambda_\mu}f_{j_1,k_1}^\mu(s)\|_{L^\infty}\|f_{j_2,k_2}^\nu(s)\|_{L^2}+2^{-8m}\big\}\\
&\lesssim 2^{m(1+\beta)}2^m\cdot 2^{-4\overline{k}}2^{-3m/2}2^{(1/2+3\beta)j_1}2^{-(1-3\beta)j_2}+2^{-4m},
\end{split}
\end{equation}
where we used Lemma \ref{PhiLocLem} and the second estimate in \eqref{LinftyBd}. This suffices to bound the contribution of the components with $j_1\leq m-20\beta m$ and $j_2\geq m-\beta m-3\overline{k}$.

On the other hand, if $j_1\geq m-20\beta m$ then, using Schur's test and Lemma \ref{Shur2Lem},
\begin{equation*}
\begin{split}
\big\|Q_{jk}T_{m,\leq l_0}[f_{j_1,k_1}^\mu,f_{j_2,k_2}^\nu]\big\|_{B_j^\sigma}&\lesssim 2^{m(1+\beta)}2^m\cdot \sup_{s\in I_m}2^{\overline{k}}(2^{l_0+\beta m}2^{10\overline{k}})^{1/2}\|\widehat{f_{j_1,k_1}^\mu}(s)\|_{L^2}\|\widehat{f_{j_2,k_2}^\nu}(s)\|_{L^2}\\
&\lesssim 2^{-\beta m-\beta j_2},
\end{split}
\end{equation*}
provided that $l_0=-\D-10\overline{k}-200\beta m$. The desired bound \eqref{gmo20} follows using also \eqref{gmo11.6}.

{\bf{Case 2.}} Assume now that
\begin{equation}\label{gmo50}
|k_1^+-k_2^+|\geq\D\,,\qquad\max(j_1,j_2)\geq m-\beta m-3\overline{k}.
\end{equation}
We may assume that $k_2^+-k_1^+\geq\D$ and, in particular $k_2\geq \D$, $|k-k_2|\leq 4$. In this case we examine the phase $\Phi_{\sigma\mu\nu}(\xi,\eta)=\Lambda_\sigma(\xi)-\Lambda_\mu(\xi-\eta)-\Lambda_\nu(\eta)$. Notice that
\begin{equation*}
\sqrt{1+a^2}+\sqrt{1+b^2}-\sqrt{1+(a+b)^2}\geq\sqrt{1+a^2}-a\geq (1+a)^{-1}/2
\end{equation*}
for any $a\leq b\in[0,\infty)$. Recalling that $\Lambda_e=\sqrt{1+d|\nabla|^2}$, $\Lambda_b=\sqrt{1+|\nabla|^2}$, $d\in(0,1)$, it is easy to see that the operator is nontrivial only when
\begin{equation}\label{gmo51}
\nu=\sigma=b,\qquad \mu=\pm e,\qquad \Phi_{\sigma\mu\nu}(\xi,\eta)=\Lambda_b(\xi)\pm\Lambda_e(\xi-\eta)-\Lambda_b(\eta).
\end{equation}

In particular, $|\nabla_\eta\Phi_{\sigma\mu\nu}(\xi,\eta)|\gtrsim 1$ in the support of the integral defining our operator. Therefore, using integration by parts in $\eta$ (Lemma \ref{tech5}), the contribution is negligible unless $\max(j_1,j_2)\geq m-\beta m$. The same $L^2\times L^\infty$ estimate as in \eqref{gmo41}, using the $L^2$ norm on the term with the higher $j$ and the $L^\infty$ norm on the term with the lower $j$, gives the desired bound unless
\begin{equation}\label{gmo51.5}
j_1\in [m-20\beta m-8k_1^+,2m]\qquad \text{ and }\qquad j_2\in [m-20\beta m-8k_2^+,2m].
\end{equation}
It remains to prove that, for $j_1$ and $j_2$ as in \eqref{gmo51.5},
\begin{equation}\label{gmo51.7}
\big\|Q_{jk}T_{m,\leq -\D-10\overline{k}-200\beta m}[f^\mu_{j_1,k_1},f^\nu_{j_2,k_2}]\big\|_{B_j^\sigma}\lesssim \bar{\eps}\,^2 2^{-5om}.
\end{equation}

Since $|\nabla_\eta\Phi_{\sigma\mu\nu}(\xi,\eta)|\gtrsim 1$ we also have stronger bounds on sublevel sets (compare with \eqref{cas4}). More precisely, combining (the proofs of) Lemma \ref{lemma00} and Lemma \ref{Shur2Lem}, we have that for any $\eps>0$
\begin{equation}\label{gmo52}
\sup_{\xi\in\mathbb{R}^3}\int_{\mathbb{R}^3}\mathbf{1}_{E_\eps}(\xi,\eta)\,d\eta+\sup_{\eta\in\mathbb{R}^3}\int_{\mathbb{R}^3}\mathbf{1}_{E_\eps}(\xi,\eta)\,d\xi\lesssim \eps 2^{3k_1^+},
\end{equation}
where, with $k\geq k_1^++\D-10$ and $\Phi_{\sigma\mu\nu}(\xi,\eta)=\Lambda_b(\xi)\pm\Lambda_e(\xi-\eta)-\Lambda_b(\eta)$ as before,
\begin{equation}\label{gmo53}
E_\eps:=\{(\xi,\eta)\in\mathbb{R}^3\times\mathbb{R}^3:\,|\xi|,|\eta|\in[2^{k-8},2^{k+8}],\,|\xi-\eta|\leq 2^{k_1+8},\,|\Phi_{\sigma\mu\nu}(\xi,\eta)|\leq\eps\}.
\end{equation}
Therefore with $l_0=-\D-10\overline{k}-200\beta m$, we can improve slightly the Schur's lemma argument:
\begin{equation*}
\begin{split}
\big\|Q_{jk}T_{m,\leq l_0}[f_{j_1,k_1}^\mu,f_{j_2,k_2}^\nu]\big\|_{B_j^\sigma}&\lesssim 2^{m(1+\beta)}2^m\cdot \sup_{s\in I_m}2^{\overline{k}}(2^{l_0}2^{3k_1^+})^{1/2}\|\widehat{f_{j_1,k_1}^\mu}(s)\|_{L^2}\|\widehat{f_{j_2,k_2}^\nu}(s)\|_{L^2}\\
&\lesssim 2^{-\beta m-\beta j_2}.
\end{split}
\end{equation*}
The desired bound \eqref{gmo20} follows in this case as well.
\end{proof}

\subsection{Space-time resonant interactions} In view of Lemmas \ref{mo2}, \ref{mo3}, \ref{mo6}, and \ref{mo7}, to complete the proof of \eqref{gmo1} it remains prove the following proposition:

\begin{proposition}\label{mo10}
For $\sigma\in\{e,b\}$ and $\mu,\nu\in\{e,b,-e,-b\}$, $\mu\neq-\nu$, we have
\begin{equation}\label{nj1}
\big\|Q_{jk}T_{m,\leq -\D}^{\sigma\mu\nu}[f_{j_1,k_1}^\mu,f_{j_2,k_2}^\nu]\big\|_{B_j^\sigma}\lesssim \bar{\eps}^2 2^{-5om},
\end{equation}
provided that
\begin{equation}\label{nj2}
k,k_1,k_2\in[-\D/2,\D/2],\qquad \max(j_1,j_2)\leq 2m,\qquad \text{ and }\qquad 3\D/2\leq j\leq m+\D.
\end{equation}
As before, we assume that $t\in[0,T]$ is fixed, $m\in[0,L+1]$, $(k,j),(k_1,j_1),(k_2,j_2)\in\mathcal{J}$, $f_\mu=f^{\alpha_1,\L_1}_{\mu}$, $f_\nu=f^{\alpha_2,\L_2}_{\nu}$, $|\L_1|+|\L_2|\leq N_1$, $|\alpha_1|+|\alpha_2|\leq 4$, and
\begin{equation*}
f_{j_1,k_1}^\mu=P_{[k_1-2,k_1+2]}Q_{j_1k_1}f_\mu,\qquad f_{j_2,k_2}^\nu=P_{[k_2-2,k_2+2]}Q_{j_2k_2}f_\nu.
\end{equation*}
\end{proposition}

The proof of this proposition contains the analysis of space-time resonances. It is more delicate than before, in the sense that we need to use the restriction operators $A_n^\sigma$ and the precise definition of the spaces $B_j^\sigma$.

We show first that we can restrict further the range of pairs $(j_1,j_2)$.

\begin{lemma}\label{Reso0}
With the hypothesis in Proposition \ref{mo10}, the bound \eqref{nj1} follows if
\begin{equation}\label{top0}
2\max(j_1,j_2)\geq (1+20\beta)[m+\min(j_1,j_2)]\qquad\text{ or }\qquad \max(j_1,j_2)\geq 14m/15.
\end{equation}
\end{lemma}

\begin{proof} Assume that $j_1\leq j_2$ and $2 j_2\geq (1+20\beta)(m+j_1)$. Then we estimate, as in \eqref{gmo41},
\begin{equation*}
\big\|Q_{jk}T_{m,\leq -\D}[f_{j_1,k_1}^\mu,f_{j_2,k_2}^\nu]\big\|_{B_j^\sigma}\lesssim 2^{m(1+\beta)}2^m\cdot \bar{\eps}^22^{-3m/2}2^{(1/2+3\beta)j_1}2^{-j_2(1-3\beta)}\lesssim \bar{\eps}^22^{-\beta m/2},
\end{equation*}
as desired. On the other hand, if
\begin{equation*}
j_2\geq 14m/15\qquad\text{ and }\qquad 2 j_2\leq (1+20\beta)(m+j_1)
\end{equation*}
then we can decompose dyadically in modulation. With $l_0:=-3m/7$ we estimate, using Schur's test as in Lemma \ref{mo7},
\begin{equation*}
\big\|Q_{jk}T_{m,\leq l_0}[f_{j_1,k_1}^\mu,f_{j_2,k_2}^\nu]\big\|_{B_j^\sigma}\lesssim 2^{m(1+\beta)}2^m\cdot \bar{\eps}^2(2^{l_0}2^{\beta m})^{1/2}2^{-j_1+3\beta j_1}2^{-j_2+3\beta j_2}\lesssim \bar{\eps}^22^{-\beta m}.
\end{equation*}
Finally, for $l\geq l_0+1$ we estimate, as in Lemma \ref{mo6},
\begin{equation*}
\big\|Q_{jk}T_{m,l}[f_{j_1,k_1}^\mu,f_{j_2,k_2}^\nu]\big\|_{B_j^\sigma}\lesssim 2^{m(1+\beta)}2^{-l}\cdot \bar{\eps}^22^{-3m/2}.
\end{equation*}
The desired conclusion follows.
\end{proof}

\begin{lemma}\label{Reso01}
With the hypothesis in Proposition \ref{mo10}, the bound \eqref{nj1} follows if
\begin{equation}\label{top01}
2\max(j_1,j_2)\leq (1+20\beta)[m+\min(j_1,j_2)]\qquad\text{ and }\qquad \max(j_1,j_2)\leq 14m/15.
\end{equation}
\end{lemma}

\begin{proof} This lemma contains the main resonant cases. We decompose dyadically in modulation and integrate by parts, using the formula \eqref{vco6.2}. It remains to prove that for any $l\in [-m+\beta m/10,-\D+4]$ and $s\in I_m$ fixed we have
\begin{equation}\label{top1}
2^{-l}\|I_{\leq l,s}[f_{j_1,k_1}^\mu(s),f_{j_2,k_2}^\nu(s)]\big\|_{B_j^\sigma}+2^{-l}\|\widetilde{I}_{l,s}[f_{j_1,k_1}^\mu(s),f_{j_2,k_2}^\nu(s)]\big\|_{B_j^\sigma}\lesssim \bar{\eps}\,^2 2^{-\beta m/5},
\end{equation}
and
\begin{equation}\label{top2}
2^{-l}2^m\|\widetilde{I}_{l,s}[(\partial_sf_{j_1,k_1}^\mu)(s),f_{j_2,k_2}^\nu(s)]\big\|_{B_j^\sigma}+2^{-l}2^m\|\widetilde{I}_{l,s}[f_{j_1,k_1}^\mu(s),(\partial_sf_{j_2,k_2}^\nu)(s)]\big\|_{B_j^\sigma}\lesssim \bar{\eps}\,^2 2^{-\beta m/5}.
\end{equation}

{\bf{Proof of \eqref{top1}.}} We notice that \eqref{top1} is an instantaneous estimate, in the sense that the time evolution plays no role. Hence, it suffices to show the following: let $\chi\in C^\infty(\mathbb{R})$ be supported in $[-1,1]$ and assume that $j, l, s, m$ satisfy
\begin{equation}\label{Ass3}
-m+\beta m/10\le l\le -\D+4,\qquad 2^{m-4}\le s\le 2^{m+4},\qquad j\leq m+\D.
\end{equation}
Define the bilinear operator $I$ by
\begin{equation}\label{gra0}
\widehat{I[f,g]}(\xi):=\int_{\mathbb{R}^3}e^{is\Phi(\xi,\eta)}\chi_l(\Phi(\xi,\eta))\widehat{f}(\xi-\eta)\widehat{g}(\eta)d\eta,\qquad \chi_l(x)=\chi(2^{-l}x),
\end{equation}
where $\Phi=\Phi_{\sigma\mu\nu}$. Assume that $f,g$ satisfy
\begin{equation}\label{gra1}
\|f\|_{\H^{N_0-N_1-5}\cap Z_1^\mu}+\|g\|_{\H^{N_0-N_1-5}\cap Z_1^\nu}\leq 1,
\end{equation}
and define $f_{j_1,k_1}:=P_{[k_1-2,k_1+2]}Q_{j_1k_1}f$, $g_{j_2,k_2}:=P_{[k_2-2,k_2+2]}Q_{j_2k_2}g$. Then
\begin{equation}\label{gra2}
2^{-l}\|Q_{jk}I[f_{j_1,k_1},g_{j_2,k_2}]\|_{B_j^\sigma}\lesssim 2^{-\beta m/5},
\end{equation}
provided that $k,k_1,k_2\in[-\D/2,\D/2]$ and $j_1,j_2$ satisfy \eqref{top01}.

In proving \eqref{gra2}, without loss of generality we may assume that $j_1\leq j_2\leq 14m/15$. With $I:=I[f_{j_1,k_1},g_{j_2,k_2}]$, recalling \eqref{psidag} and \eqref{cas2}, we will show that
\begin{equation}\label{Alx81}
2^{-l}\sup_{|\xi|\in[2^{-3\D/4},2^{3\D/4}]} \vert (1+2^{m}\Psi^\dagger_\sigma(\xi))^{1/2+10\beta}\widehat{I}(\xi)\vert\lesssim 2^{2\beta m-m/2}.
\end{equation}
Notice that this is stronger than the bound \eqref{gra2}. Indeed if $\sigma=b$ then for $j$ fixed we estimate
\begin{equation*}
\begin{split}
\sup_{0\leq n\leq j+1}&2^{(1+\beta)j}2^{-4\beta n}\big\|A_{n,(j)}^\sigma Q_{jk}I\big\|_{L^2}\\
&\lesssim \sup_{0\leq n\leq j+1}2^{(1+\beta)j}2^{-4\beta n}\big\|\varphi_{-n}^{[-j-1,0]}(\Psi^\dagger_\sigma(\xi))\varphi_{k}(\xi)\widehat{I}(\xi)\big\|_{L^2_\xi}\\
&\lesssim \sum_{n\geq 0}2^{(1+\beta)j}2^{-n/2-4\beta \min(n,j)}\big\|\varphi_{-n}^{(-\infty,0]}(\Psi^\dagger_b(\xi))\varphi_{k}(\xi)\widehat{I}(\xi)\big\|_{L^\infty_\xi},
\end{split}
\end{equation*}
and notice that \eqref{gra2} would follow from \eqref{Alx81}. The proof if similar (in fact simpler) if $\sigma=e$.

To prove \eqref{Alx81} assume that $m\geq \D^2$ and $\xi\in\mathbb{R}^3$ is fixed with $|\xi|\in[2^{-3\D/4},2^{3\D/4}]$. Let
\begin{equation*}
\Xi(\xi,\eta):=(\nabla_\eta\Phi)(\xi,\eta).
\end{equation*}
We remove first the nonresonant contribution. With $\kappa_r:=2^{\beta m/40}\big(2^{-m/2}+2^{j_2-m}\big)$ we define
\begin{equation}\label{grn1}
\mathcal{NR}(\xi):=\int_{\mathbb{R}^3}e^{is\Phi(\xi,\eta)}\chi_{l}(\Phi(\xi,\eta))(1-\varphi(\kappa_r^{-1}\Xi(\xi,\eta)))\widehat{f_{j_1,k_1}}(\xi-\eta)\widehat{g_{j_2,k_2}}(\eta)d\eta.
\end{equation}
With $\psi_1:=\varphi_{\leq m-\beta m/20}$ and $\psi_2:=1-\varphi_{\leq m-\beta m/20}$, we further decompose
\begin{equation*}
\begin{split}
&\mathcal{NR}(\xi)=\mathcal{NR}_1(\xi)+\mathcal{NR}_2(\xi),\\
&\mathcal{NR}_i(\xi):=C2^{l}\int_{\mathbb{R}}\int_{\mathbb{R}^2}e^{i(s+\lambda)\Phi(\xi,\eta)}\widehat{\chi}(2^l\lambda)\psi_i(\lambda)(1-\varphi(\kappa_r^{-1}\Xi(\xi,\eta)))\widehat{f_{j_1,k_1}}(\xi-\eta)\widehat{g_{j_2,k_2}}(\eta)\,d\eta d\lambda.
\end{split}
\end{equation*}
Since $\widehat{\chi}$ is rapidly decreasing we have $\|\varphi_k\cdot\mathcal{NR}_2\|_{L^\infty}\lesssim 2^{-4m}$, which gives an acceptable contribution. On the other hand, in the support of the integral defining $\mathcal{NR}_1$, we have that $\vert s+\lambda\vert\approx 2^m$ and integration by parts in $\eta$ (using Lemma \ref{tech5}) gives $\|\varphi_k\cdot\mathcal{NR}_1\|_{L^\infty}\lesssim 2^{-4m}$. Therefore the contribution of $\mathcal{NR}$ can be estimated as claimed in \eqref{Alx81}.

In view of Proposition \ref{spaceres} (ii), (iii), $\widehat{I}-\mathcal{NR}$ is nontrivial only if we have a space-time resonance. In particular, we may assume that
\begin{equation}\label{Alx74.6}
(\sigma,\mu,\nu)\in\{(b,e,e),(b,e,b),(b,b,e)\},\qquad \min\big(\big||\xi|-\gamma_1\big|,\big||\xi|-\gamma_2\big|\big)\leq 2^{-\D/2}.
\end{equation}
We may also assume that $|\xi_3|\geq 2^{-\D/2}$ (the proof is similar if $|\xi_1|\geq 2^{-\D/2}$ or if $|\xi_2|\geq 2^{-\D/2}$). By rotation, using the vector-fields $\Omega_1$ and $\Omega_2$ we may assume that $\xi=(0,0,\xi_3)$. We would like to use Lemma \ref{RotIBP}. Recalling now the definition of $\lambda_\mu$ in \eqref{deflambd}, we let
\begin{equation}\label{nba1}
\begin{split}
&\Phi^1(\xi,\eta):=(\Omega_1)_\eta\Phi(\xi,\eta)=\frac{\lambda'_\mu(|\xi-\eta|)}{|\xi-\eta|}(\eta_2\xi_3-\eta_3\xi_2),\\
&\Phi^2(\xi,\eta):=(\Omega_2)_\eta\Phi(\xi,\eta)=\frac{\lambda'_\mu(|\xi-\eta|)}{|\xi-\eta|}(\eta_3\xi_1-\eta_1\xi_3),\\
&\Phi^3(\xi,\eta):=(\Omega_3)_\eta\Phi(\xi,\eta)=\frac{\lambda'_\mu(|\xi-\eta|)}{|\xi-\eta|}(\eta_1\xi_2-\eta_2\xi_1).
\end{split}
\end{equation}
Let $\kappa_\theta:=2^{\beta m/40}2^{-m/2}$ and define
\begin{equation*}
\begin{split}
\mathcal{R}_{\perp}(\xi):=\int_{\mathbb{R}^3}&e^{is\Phi(\xi,\eta)}\chi_{l}(\Phi(\xi,\eta))\varphi(\kappa_r^{-1}\Xi(\xi,\eta))\\
&\big[1-\varphi(\kappa_\theta^{-1}\Phi^1(\xi,\eta))\varphi(\kappa_\theta^{-1}\Phi^2(\xi,\eta))\big]\widehat{f_{j_1,k_1}}(\xi-\eta)\widehat{g_{j_2,k_2}}(\eta)d\eta.
\end{split}
\end{equation*}
We apply Lemma \ref{RotIBP} twice, after decomposing
\begin{equation*}
1-\varphi(\kappa_\theta^{-1}\Phi^1(\xi,\eta))\varphi(\kappa_\theta^{-1}\Phi^2(\xi,\eta))=[1-\varphi(\kappa_\theta^{-1}\Phi^2(\xi,\eta))]+\varphi(\kappa_\theta^{-1}\Phi^2(\xi,\eta))[1-\varphi(\kappa_\theta^{-1}\Phi^1(\xi,\eta))].
\end{equation*}
Notice that the factors $\psi_1(\xi,\eta)$, $\psi_2(\xi,\eta)$ are already accounted for by the factor $\varphi(\kappa_r^{-1}\Xi(\xi,\eta))$ and the assumptions $|\xi_3|\geq 2^{-\D/2}$ and $m\geq \D^2$. It follows that $|\mathcal{R}_{\perp}(\xi)|\lesssim 2^{-4m}$.

It remains to bound the resonant component
\begin{equation}\label{nba2}
\begin{split}
\mathcal{R}_{||}(\xi):=J_{||}[f_{j_1,k_1},&g_{j_2,k_2}](\xi):=\int_{\mathbb{R}^3}e^{is\Phi(\xi,\eta)}\chi_{l}(\Phi(\xi,\eta))\varphi(\kappa_r^{-1}\Xi(\xi,\eta))\\
&\varphi(\kappa_\theta^{-1}\Phi^1(\xi,\eta))\varphi(\kappa_\theta^{-1}\Phi^2(\xi,\eta))\widehat{f_{j_1,k_1}}(\xi-\eta)\widehat{g_{j_2,k_2}}(\eta)d\eta.
\end{split}
\end{equation}
More precisely, for \eqref{Alx81} it remains to prove that if $\xi=(0,0,\xi_3)$, $\xi_3\in[2^{-\D/2},2^{\D/2}]$, then
\begin{equation}\label{nba4}
\vert (1+2^{m}\Psi^\dagger_b(\xi))\mathcal{R}_{||}(\xi)\vert\lesssim 2^{2\beta m-m/2}2^l.
\end{equation}

We examine now the integral in \eqref{nba2}. In view of Proposition \ref{spaceres} (ii), this integral is nontrivial only if
\begin{equation}\label{EstimPsi}
\vert \Psi_b(\xi)\vert =\vert \Phi(\xi,p(\xi))\vert\lesssim\vert \Phi(\xi,\eta)\vert+\vert \Phi(\xi,\eta)-\Phi(\xi, p(\xi))\vert\lesssim 2^l+\kappa_r^2.
\end{equation}
Using \eqref{nba1} and Proposition \ref{spaceres} (ii), for $\xi=(0,0,\xi_3)$ fixed, $\eta$ is supported in the rectangle
\begin{equation}\label{nba5}
\mathcal{Q}_{\xi}:=\{\eta=(\eta_1,\eta_2,\eta_3):\,|\eta_1|+|\eta_2|\leq 2^{4\D}\kappa_\theta,\,|\eta_3-p_+(\xi_3)|\leq 2^{4\D}\kappa_r\}.
\end{equation}

Recall from Lemma \ref{LinEstLem} (ii) and \eqref{gra1} that
\begin{equation}\label{nba6}
\begin{split}
2^{j_1/2-j_1/20}\Vert \widehat{f_{j_1,k_1}}\Vert_{L^\infty}+2^{j_1-j_1/20}\Vert\sup_{\theta\in\mathbb{S}^2} |\widehat{f_{j_1,k_1}}(r\theta)|\Vert_{L^2(r^2dr)}&\lesssim 1,\\
2^{j_2/2-j_2/20}\Vert \widehat{g_{j_2,k_2}}\Vert_{L^\infty}+2^{j_2-j_2/20}\Vert\sup_{\theta\in\mathbb{S}^2} |\widehat{g_{j_2,k_2}}(r\theta)|\Vert_{L^2(r^2dr)}&\lesssim 1.
\end{split}
\end{equation}
Using only the $L^\infty$ bounds in \eqref{nba6} and ignoring the cutoff function $\chi_{l}(\Phi(\xi,\eta))$ in \eqref{nba2}, we estimate first
\begin{equation*}
|\mathcal{R}_{||}(\xi)|\lesssim \kappa_r\kappa_\theta^22^{-9j_2/20}2^{-9j_1/20}\lesssim 2^{\beta m/10}2^{-9j_2/20}2^{-9j_1/20}2^{-m}(2^{-m/2}+2^{j_2-m}).
\end{equation*}
Since $\vert \Psi_b(\xi)\vert \lesssim 2^l+\kappa_r^2$ (see \eqref{EstimPsi}), the desired bound \eqref{nba4} follows easily if $j_2\leq m/2$. On the other hand, if $j_2\geq m/2$ then the left-hand side of \eqref{nba4} is dominated by
\begin{equation*}
C2^m(2^l+\kappa_r^2)\cdot 2^{\beta m/10}2^{-9j_2/20}2^{-9j_1/20}2^{-m}2^{j_2-m}\lesssim (2^l+\kappa_r^2)2^{\beta m/2}2^{-m}2^{11j_2/20-9j_1/20}.
\end{equation*}
In view of the assumption \eqref{top01}, $11j_2/20-9j_1/20\leq 3m/10-10\beta m$. The desired bound \eqref{nba4} follows if $\kappa_r^2\leq 2^l2^{m/5}$.

Finally assume that $\kappa_r^2\ge 2^l2^{m/5}$ (in particular $j_2\geq 11m/20$). In this case the restriction $|\Phi(\xi,\eta)|\lesssim 2^l$ is stronger and we have to use it. We decompose, with $p_-:=\lfloor\log_2(2^{l/2}\kappa_r^{-1})+\D\rfloor$,
\begin{equation*}
\mathcal{R}_{||}(\xi)=\sum_{p\in[p_-,0]}\mathcal{R}^p_{||}(\xi),
\end{equation*}
where
\begin{equation}\label{nba7.5}
\begin{split}
\mathcal{R}_{||}^p(\xi):=J^p_{||}[f_{j_1,k_1},&g_{j_2,k_2}](\xi):=\int_{\mathbb{R}^3}e^{is\Phi(\xi,\eta)}\chi_{l}(\Phi(\xi,\eta))\varphi_p^{[p_-,1]}(\kappa_r^{-1}\Xi(\xi,\eta))\\
&\varphi(\kappa_\theta^{-1}\Phi^1(\xi,\eta))\varphi(\kappa_\theta^{-1}\Phi^2(\xi,\eta))\widehat{f_{j_1,k_1}}(\xi-\eta)\widehat{g_{j_2,k_2}}(\eta)d\eta.
\end{split}
\end{equation}
Notice that if $\mathcal{R}_{||}^p(\xi)\neq 0$ then $\vert\Psi_b(\xi)\vert\lesssim 2^{2p}\kappa_r^2$ (this is stronger than \eqref{EstimPsi}). The term $\mathcal{R}_{||}^{p_-}(\xi)$ can be bounded as before. On the other hand, for $p\geq p_--1$ we would like to get a more precise description on the support of integration in $\eta$ (better than the one in \eqref{nba5}). For this we write
\begin{equation}\label{nba8}
\Phi(\xi,\eta)=\sqrt{1+|\xi|^2}-\sqrt{1+d_\mu|\xi-\eta|^2}-\sqrt{1+d_\nu|\eta|^2},
\end{equation}
where $d_e=d\in(0,1)$ and $d_b=1$. Since $\xi=(0,0,\xi_3)$, $\xi_3\in[2^{-\D/2},2^{\D/2}]$, and $|\eta_1|+|\eta_2|\leq 2^{4\D}\kappa_\theta$, the condition $|\Xi(\xi,\eta)|\in[2^{p-2}\kappa_r,2^{p+2}\kappa_r]$ implies that $|\partial_{\eta_3} \Phi(\xi,\eta)|\approx 2^p\kappa_r$. In particular, using Proposition \ref{spaceres} (ii), the $\eta$ support of integration is included in the set
\begin{equation*}
\{\eta=(\eta_1,\eta_2,\eta_3):\,|\eta_1|+|\eta_2|\leq 2^{4\D}\kappa_\theta,\,|\eta_3-p_+(\xi_3)|\approx 2^p\kappa_r,\,|\Phi(\xi,\eta)|\leq 2^l\}.
\end{equation*}
Based on Lemma \ref{lemma00}, this set is essentially contained in a union of two $(\kappa_\theta)^2\times 2^l2^{-p}\kappa_r^{-1}$ tubes. Using \eqref{nba6} and estimating $\Vert \widehat{f_{j_1,k_1}}\Vert_{L^\infty}\lesssim 2^{-9j_1/20}\lesssim 2^{40\beta m}2^{9(m-2j_2)/20}$, see \eqref{top01}, we have
\begin{equation*}
\begin{split}
|\mathcal{R}^p_{||}(\xi)|&\lesssim (\kappa_\theta)^2\times (2^l2^{-p}\kappa_r^{-1})^{1/2}\Vert \widehat{g_{j_2,k_2}}\Vert_{L^\infty_\theta L^2(rdr)}\Vert \widehat{f_{j_1,k_1}}\Vert_{L^\infty}\\
&\lesssim (\kappa_\theta)^2\times (2^l2^{-p}\kappa_r^{-1})^{1/2}2^{40\beta m}2^{9m/20}2^{-9j_2/5}.
\end{split}
\end{equation*}
Therefore, since $|\Psi(\xi)|\lesssim 2^{2p}\kappa_r^2$ in the support of $\mathcal{R}^p_{||}$,
\begin{equation*}
\begin{split}
\vert (1+2^{m}\Psi_b(\xi))\mathcal{R}^p_{||}(\xi)\vert&\lesssim 2^{m+2p}\kappa_r^2\cdot 2^{-m+42\beta m}(2^l2^{-p}\kappa_r^{-1})^{1/2}2^{9m/20}2^{-9j_2/5}\\
&\lesssim 2^{3p/2}2^{l/2}2^{-m}2^{-j_2/5}.
\end{split}
\end{equation*}
This suffices to prove \eqref{nba4} since $2^p\leq 1$, $2^{-l/2}\leq 2^{m/2}$, and $2^{-j_2/5}\leq 2^{-m/10}$. This completes the proof of the main bound \eqref{top1}.

{\bf{Proof of \eqref{top2}.}} As in \eqref{gra2}, it suffices to prove that
\begin{equation}\label{grd2}
2^{-l}\|Q_{jk}I[F_{j_1,k_1},g_{j_2,k_2}]\|_{B_j^\sigma}+2^{-l}\|Q_{jk}I[f_{j_1,k_1},G_{j_2,k_2}]\|_{B_j^\sigma}\lesssim 2^{-\beta m/5},
\end{equation}
where $I$ is defined as in \eqref{gra0}, $F=\bar{\eps}^{-1}2^m\partial_s f_\mu$, and $G:=\bar{\eps}^{-1}2^m\partial_s g_\nu$. The functions $f,g,F,G$ satisfy the bounds
\begin{equation}\label{grd3}
\begin{split}
&\|f\|_{\H^{N_0-N_1-5}\cap Z_1^\mu}+\|g\|_{\H^{N_0-N_1-5}\cap Z_1^\nu}\leq 1,\\
&\|F\|_{\H^{N_0-N_1-6}}+2^{m/2}\|F\|_{L^2}+\|G\|_{\H^{N_0-N_1-6}}+2^{m/2}\|G\|_{L^2}\leq 1,
\end{split}
\end{equation}
compare with the bounds in Proposition \ref{sDeriv} (iii). As before, we may assume that $k_1,k_2\in[-\D/2,\D/2]$, and that the parameters $j, l, s, m, j_1,j_2$ satisfy the bounds \eqref{Ass3} and \eqref{top01}.

As before, for \eqref{grd2} it suffices to prove the stronger pointwise bound
\begin{equation*}
\begin{split}
2^{-l}&\sup_{|\xi|\in[2^{-3\D/4},2^{3\D/4}]}\big|(1+2^m\Psi_{\sigma}^\dagger(\xi))^{1/2+10\beta}\mathcal{F}\{I[F_{j_1,k_1},g_{j_2,k_2}]\}\big|\\
&+2^{-l}\sup_{|\xi|\in[2^{-3\D/4},2^{3\D/4}]}\big|(1+2^m\Psi_{\sigma}^\dagger(\xi))^{1/2+10\beta}\mathcal{F}\{I[f_{j_1,k_1},G_{j_2,k_2}]\}\big|\lesssim 2^{-m/2}.
\end{split}
\end{equation*}
In proving this we may assume $j_1\leq j_2$, $m\geq \D^2$, and first remove the negligible nonresonant interactions (defined as in \eqref{grn1}). Then we may assume that $\sigma=b$, $\xi=(0,0,\xi_3)$, with $\xi_3\in [2^{-\D/2},2^{\D/2}]$, and remove the negligible non-parallel interactions. After these reductions, with $J_{||}$ defined as in \eqref{nba2}, it remains to prove that
\begin{equation}\label{grd4}
\begin{split}
\big|(1+2^m\Psi_{b}^\dagger(\xi))^{1/2+10\beta}&J_{||}[F_{j_1,k_1},g_{j_2,k_2}](\xi)\big|\\
&+\big|(1+2^m\Psi_{b}^\dagger(\xi))^{1/2+10\beta}J_{||}[f_{j_1,k_1},G_{j_2,k_2}](\xi)\big|\lesssim 2^l2^{-m/2}.
\end{split}
\end{equation}

The functions $f_{j_1,k_1}$ and $g_{j_2,k_2}$ satisfy the bounds \eqref{nba6}. Moreover,
\begin{equation}\label{grd5}
\Vert \sup_{\theta\in\mathbb{S}^2}|\widehat{F_{j_1,k_1}}(r\theta)|\Vert_{L^2(r^2dr)}+\Vert \sup_{\theta\in\mathbb{S}^2}| \widehat{G_{j_2,k_2}}(r\theta)|\Vert_{L^2(r^2dr)}\lesssim 2^{-m/2+m/40}.
\end{equation}
as a consequence of \eqref{grd3}, using the same interpolation argument as in the proof of \eqref{RadL2}.
We ignore first the cutoff function $\chi_l(\Phi(\xi,\eta))$ and notice that the variable $\eta$ is included in the set $\mathcal{Q}_\xi$ defined in \eqref{nba5}. Using \eqref{grd5} and the $L^\infty$ bounds in \eqref{nba6} we estimate first
\begin{equation}\label{grd7}
\begin{split}
|J_{||}[F_{j_1,k_1},g_{j_2,k_2}](\xi)|&+|J_{||}[f_{j_1,k_1},G_{j_2,k_2}](\xi)|\lesssim \kappa_\theta^2\kappa_r^{1/2} 2^{-j_1/2+j_1/20}2^{-m/2+m/40}\\
&\lesssim 2^{-3m/2+m/39}2^{-9j_1/20}(2^{-m/4}+2^{(j_2-m)/2}).
\end{split}
\end{equation}
Since $\kappa_r=2^{\beta m/40}(2^{-m/2}+2^{j_2-m})$ and $\vert \Psi_b(\xi)\vert \lesssim 2^l+\kappa_r^2$ (see \eqref{EstimPsi}), the desired bound \eqref{grd4} follows easily from \eqref{grd7} if $j_2\leq m/2$. On the other hand, if $j_2\geq m/2$ then $2^{-9j_1/20}\lesssim 2^{40\beta m}2^{9/20(m-2j_2)}$, and the bound \eqref{grd7} gives
\begin{equation}\label{grd8}
|J_{||}[F_{j_1,k_1},g_{j_2,k_2}](\xi)|+|J_{||}[f_{j_1,k_1},G_{j_2,k_2}](\xi)|\lesssim 2^{-3m/2}2^{-2j_2/5}.
\end{equation}
The desired bound \eqref{grd4} follows if $\kappa_r^2\leq 2^l2^{2j_2/5}$.

On the other hand, if $\kappa_r^2\geq 2^l2^{2j_2/5}$ (in particular this implies $j_2\geq 11m/20$) then we have to use the stronger restriction $|\Phi(\xi,\eta)|\lesssim 2^l$. For $p\in[p_-,0]$, $p_-:=\lfloor\log_2(2^{l/2}\kappa_r^{-1})+\D\rfloor$, we define the operators $J^p_{||}$ as in \eqref{nba7.5}. Notice that the contribution of $J^{p_-}_{||}$ can be estimated easily using the fact that $\vert \Psi_b(\xi)\vert \lesssim 2^l$ in the support of $J^{p_-}_{||}$. Moreover, as proved earlier, the $\eta$ support of integration in the definition of $J^{p_-}_{||}$ is included in the set
\begin{equation*}
\{\eta=(\eta_1,\eta_2,\eta_3):\,|\eta_1|+|\eta_2|\leq 2^{4\D}\kappa_\theta,\,|\eta_3-p_+(\xi_3)|\approx 2^p\kappa_r,\,|\Phi(\xi,\eta)\leq 2^l\},
\end{equation*}
which is essentially contained in a union of two $(\kappa_\theta)^2\times 2^l2^{-p}\kappa_r^{-1}$ tubes (based again 
on Lemma \ref{lemma00}). Using \eqref{grd5} and the $L^\infty$ bounds in \eqref{nba6} we estimate
\begin{equation*}
|J_{||}^p[F_{j_1,k_1},g_{j_2,k_2}](\xi)|+|J_{||}^p[f_{j_1,k_1},G_{j_2,k_2}](\xi)|\lesssim \kappa_\theta^2(2^l2^{-p}\kappa_r^{-1})^{1/2} 2^{-9j_1/20}2^{-m/2+m/40}.
\end{equation*}
Since $2^{-9j_1/20}\lesssim 2^{40\beta m}2^{9/20(m-2j_2)}$ and $\vert \Psi_b(\xi)\vert \lesssim 2^{2p}\kappa_r^2$, it follows that
\begin{equation*}
\begin{split}
\big|(1+2^m\Psi_{b}^\dagger(\xi))^{1/2+10\beta}&J^p_{||}[F_{j_1,k_1},g_{j_2,k_2}](\xi)\big|+\big|(1+2^m\Psi_{b}^\dagger(\xi))^{1/2+10\beta}J^p_{||}[f_{j_1,k_1},G_{j_2,k_2}](\xi)\big|\\
&\lesssim (2^{m+2p}\kappa_r^2)^{1/2+10\beta}\cdot 2^{-m}(2^l2^{-p}\kappa_r^{-1})^{1/2} 2^{9/20(m-2j_2)}2^{-m/2+m/38}\\
&\lesssim 2^{p/2}2^{l/2}2^{-2j_2/5}2^{-m}.
\end{split}
\end{equation*}
The desired conclusion \eqref{grd4} follows, which completes the proof of the lemma.
\end{proof}

\end{document}